\documentclass[11.5pt]{amsart}
\calclayout
 \usepackage{amssymb,amsmath,amsfonts,epsfig,latexsym,tikz}
  \usepackage[alphabetic]{amsrefs}
 \usepackage{tikz-cd}
\usepackage{hyperref}
\usepackage{enumerate}
\usepackage{mathtools}
\usepackage{verbatim}
\usepackage{cleveref}
\usepackage[shortlabels]{enumitem}
\usepackage{subcaption}
\usepackage{ stmaryrd }

\usetikzlibrary{positioning}
\usetikzlibrary{matrix}
\usetikzlibrary{decorations}
\usetikzlibrary{decorations.pathreplacing, decorations.pathmorphing, angles,quotes}
 
 \newtheorem{theorem}{Theorem}[section]

\newtheorem{proposition}[theorem]{Proposition}
\newtheorem{lemma}[theorem]{Lemma}
\newtheorem{corollary}[theorem]{Corollary}

\makeatletter
\newcommand{\newreptheorem}[2]{\newtheorem*{rep@#1}{\rep@title}\newenvironment{rep#1}[1]{\def\rep@title{#2 \ref*{##1}}
\begin{rep@#1}}{\end{rep@#1}}}
\makeatother
\newreptheorem{theorem}{Theorem}

\newcommand{\dom}{\mathrm{dom}}
\newcommand{\del}{\mathrm{del}}
\theoremstyle{definition}
\newtheorem{remark}[theorem]{Remark}
\newtheorem{claim}[theorem]{Claim}
\newtheorem{example}[theorem]{Example}

\begin{document}

\title{Fine multidegrees, universal Gr\"{o}bner bases, and matrix Schubert varieties}          
    
\author{Daoji Huang} 
\address{(Daoji Huang) University of Massachusetts Amherst}
\email{daojihuang@umass.edu}

\author{Matt Larson}
\address{(Matt Larson) Princeton University and the Institute for Advanced Study}
\email{mattlarson@princeton.edu}
\date{\today}
\begin{abstract}
We give a criterion for a collection of polynomials to be a universal Gr\"{o}bner basis for an ideal in terms of the multidegree of the closure of the corresponding affine variety in $(\mathbb{P}^1)^N$. This criterion can be used to give simple proofs of several existing results on universal Gr\"{o}bner bases. 
We introduce fine Schubert polynomials, which record the multidegrees of the closures of  matrix Schubert varieties in $(\mathbb{P}^1)^{n^2}$. 
We compute the fine Schubert polynomials of permutations $w$ where the coefficients of the Schubert polynomials of $w$ and $w^{-1}$ are all either 0 or 1, and we use this to give a universal Gr\"{o}bner basis for the ideal of the matrix Schubert variety of such a permutation. 
\end{abstract}
 \subjclass[2020]{Primary: 14M15, 13P10, 13C40, 14M12, 05E14}

\maketitle
\tableofcontents
\vspace{-20 pt}

\section{Introduction}
\label{sec:intro}

For a permutation $w \in S_n$, let $X_w$ be the associated matrix Schubert variety, i.e., the closure in the space of $n \times n$ complex matrices of the preimage of the Schubert cell labeled by $w$ under the map from $GL_n$ to the flag variety $GL_n/B$.  These varieties, which were introduced by Fulton \cite{FultonDegen}, can be used to give algebraic proofs of geometric facts about Schubert varieties and to give an algebro-geometric interpretation of Schubert polynomials \cite{KM05}.
We identify the space of $n \times n$ complex matrices with $\mathbb{A}^{n^2}$.

Gr\"{o}bner degenerations are a powerful tool to study the algebra, combinatorics, and geometry of matrix Schubert varieties \cites{KM05,KMY09,HPW22,KW21}.
Let $\varphi \colon \mathbb{G}_m \to \mathbb{G}_m^{n^2}$ be an inclusion of a $1$-parameter subgroup into the torus which acts on $\mathbb{A}^{n^2}$. 
A Gr\"{o}bner degeneration of $X_w$ can be understood as the limit $\lim_{t \to 0} \varphi(t) \cdot X_w$, where $t \in \mathbb{G}_m$. 
If $\varphi$ is sufficiently generic, then the limit will be a scheme whose reduction is a union of coordinate subspaces, and so the degeneration is an essentially combinatorial object. Properties of $X_w$ (such as its Hilbert function or Cohen--Macaulayness) can be analyzed in this degeneration. See \cite[Section 1.7]{KM05}.

We introduce an invariant which controls the possible Gr\"{o}bner degenerations of $X_w$, and can, in some cases, be used to produce a universal Gr\"{o}bner basis for the ideal of $X_w$. This invariant can be defined for any pure dimensional closed subscheme of an affine space, so we work in that level of generality. We work over the complex numbers for convenience, but all statements in this paper can be easily adapted to hold over any field. See Remark~\ref{rem:characteristic}.

\medskip

Let $X \subseteq \mathbb{A}^N$ be a closed subscheme whose irreducible components all have codimension $d$. Consider the natural inclusion of $\mathbb{A}^N$ into $(\mathbb{P}^1)^N$, and let $\overline{X}$ denote the closure of $X$ in $(\mathbb{P}^1)^N$. Then $\overline{X}$ defines a class $[\overline{X}] \in H_{2(N - d)}((\mathbb{P}^1)^N)$ in the homology of $(\mathbb{P}^1)^N$. Let $[N] = \{1, \dotsc, N\}$. For each $i \in [N]$, let $[H_{i}] \in H^2((\mathbb{P}^1)^{N})$ be the class of the closure of the hyperplane $\{(a_{k})_{k \in [N]} : a_{i} = 0\}$ in $(\mathbb{P}^1)^{N}$. We say that a polynomial is \textbf{squarefree-supported} if each monomial appearing is a product of distinct variables. There is a unique squarefree-supported polynomial $\mathfrak{F}_{\overline{X}}(z) \in \mathbb{Z}[z_{i} : i \in [N]]$  such that
$$[\overline{X}] = \mathfrak{F}_{\overline{X}}([H_{1}], \dotsc, [H_N]) \smallfrown [(\mathbb{P}^1)^{N}].$$
We call this polynomial the \textbf{fine multidegree polynomial} of $X$ because it records information about a very fine grading:
it can be described in terms of the Hilbert function of the $\mathbb{Z}^{N}$-graded homogenization of the ideal of $X$, see Section~\ref{sec:grobner}. Note that $\mathfrak{F}_{\overline{X}}(z)$ is homogeneous of degree $d$. 

There is a canonical basis for $H_*((\mathbb{P}^1)^{N})$, given by the classes of coordinate subspaces. The basis for $H_{2k}((\mathbb{P}^1)^{N})$ is labeled by subsets $S$ of $[N]$ of size $N - k$, with $S$ corresponding to the class of the coordinate subspace $\{(a_{s})_{s \in [N]} : a_{i} = 0 \text{ if }i \in S\}$. 
The \emph{multidegree} of $\overline{X}$ is the collection of coefficients which are used to expand the class $[\overline{X}]$ in terms of this canonical basis. These coefficients are recorded by the coefficients of the fine multidegree polynomial: the coefficient of $z^S := \prod_{i \in S} z_{i}$ is the coefficient of the basis element corresponding to $S$. 

If $|S| = d$, the coefficient of $z^S$ is also the degree of the composition $X \hookrightarrow \mathbb{A}^N \to \mathbb{A}^{S^c}$, where the second map is the projection onto the factors labeled by elements not in $S$. In particular, the coefficients of the fine multidegree polynomial are nonnegative. 

The polynomial $\mathfrak{F}_{\overline{X}}(z)$ controls the possible Gr\"{o}bner degenerations of $X$. Indeed, $\lim_{t \to 0} \varphi(t) \cdot \overline{X}$ contains $\lim_{t \to 0} \varphi(t) \cdot X$. The multidegree of $\lim_{t \to 0} \varphi(t) \cdot \overline{X}$ is the same as the multidegree of $\overline{X}$, so if the coordinate subspace corresponding to $S \subset [N]$ is an irreducible component of the limit, then the coefficient of $z^S$ in $\mathfrak{F}_{\overline{X}}(z)$ must be positive. 
One can show that the coefficient of $z^S$ in the fine multidegree polynomial is equal to the maximal multiplicity of $\{(a_{k})_{k \in [N]} : a_{i} = 0 \text{ if }i \in S\}$ in any Gr\"{o}bner degeneration of $X$. See Proposition~\ref{prop:finegrobner}.

The \textbf{support} of a fine multidegree polynomial is the collection of subsets $S$ of $[N]$ where the coefficient of $z^S$ in the fine multidegree polynomial is nonzero.  If $X$ is integral, the support of a fine multidegree polynomial is the set of bases of a \emph{matroid} of rank $d$ on $[N]$. This matroid is dual to the \emph{algebraic matroid} of $X$: if $|S| = d$, the coefficient of $z^S$ is nonzero if and only if the coordinate functions $\{a_{i} : i \in S^c\}$ on $X$ are a transcendence basis for the function field of $X$. 

\medskip

A collection of polynomials in an ideal is said to be a \emph{universal Gr\"{o}bner basis} for that ideal if they are a Gr\"{o}bner basis for any term order. We follow \cite{SturmfelsGrobner} for Gr\"{o}bner basis conventions. While general results guarantee the existence of a universal Gr\"{o}bner basis, there are only a few families of ideals where explicit universal Gr\"{o}bner bases are known. The maximal minors of a generic $p \times q$ matrix form a universal Gr\"{o}bner basis \cite{BernsteinZelevinsky,SturmfelsZelevinsky}, see also \cite{CDGMaximal}. This remains true if some of the entries of the $p \times q$ matrix are set to $0$ \cite[Proposition 5.4]{BoocherDeterminant}. There is also a known universal Gr\"{o}bner basis for the ideal of $2 \times 2$ minors, which is a special case of \cite[Propositions 4.11 and 8.11]{SturmfelsGrobner} or \cite[Proposition 10.1.11]{Villarreal}. See \cite[Chapter 5]{BCRV} for a summary. In some cases, ideals associated to linear spaces \cite[Theorem 4]{PS} \cite[Theorem 1.3]{ArdilaBoocher} \cite[Theorem B]{BergetFink} have explicit universal Gr\"{o}bner bases.

We give a criterion for a collection of polynomials to be a universal Gr\"{o}bner basis for an ideal. This criterion can be used to give a simple proof of all of the above results. Let $I_X$ denote the ideal of $X$ in $\mathbb{C}[a_1, \dotsc, a_N]$. The \textbf{spread} of $f \in I_X$ is the collection of $i \in [N]$ such that $a_i$ appears in a monomial with nonzero coefficient in $f$. We denote this set by $\operatorname{spr}(f)$. 
For nonzero $f_1, \dotsc, f_r \in I_X$, let $\Delta(f_1, \dotsc, f_r)$ be the simplicial complex on $[N]$ whose nonfaces are generated by the spreads of the $f_i$.

\begin{theorem}\label{thm:irreducible}
Let $X$ be a closed subscheme of $\mathbb{A}^N$ whose irreducible components all have codimension $d$.
Let $f_1, \dotsc, f_r \in I_X$ be nonzero squarefree-supported polynomials. Suppose that $\Delta(f_1, \dotsc, f_r)$ is pure of dimension $N - d -1$, and that for each facet $U$ of $\Delta(f_1, \dotsc, f_r)$, the coefficient of $z^{U^c}$ in $\mathfrak{F}_{\overline{X}}(z)$ is positive.
Then $\{f_1, \dotsc, f_r\}$ is a universal Gr\"{o}bner basis for ${I}_X$, and all coefficients of $\mathfrak{F}_{\overline{X}}(z)$ are equal to either $0$ or $1$. 
\end{theorem}

See Theorem~\ref{thm:grobnercriterion} for a more general criterion, which is valid for any closed subscheme of $\mathbb{A}^N$. If $X$ is integral and all coefficients of $\mathfrak{F}_{\overline{X}}(z)$ are $0$ or $1$, then there are $f_1, \dotsc, f_r \in I_X$ such that the hypothesis of Theorem~\ref{thm:irreducible} holds. See Remark~\ref{rem:CS}.
It is automatic that if the coefficient of $z^{U^c}$ is positive, then $U$ does not contain the spread of any nonzero $f \in I_X$. See Lemma~\ref{lem:spreadprojection}. If $f_1, \dotsc, f_r$ is a universal Gr\"{o}bner basis (or even a Gr\"{o}bner basis with respect to any lexicographic order), then $\Delta(f_1, \dotsc, f_r)$ is pure of dimension $N - d - 1$, and for each facet $U$ of $\Delta(f_1, \dotsc, f_r)$, the coefficient of $z^{U^c}$ in $\mathfrak{F}_{\overline{X}}(z)$ is positive. See Proposition~\ref{prop:pure}. 

For a collection $\{f_1, \dotsc, f_r\}$ satisfying the hypothesis of Theorem~\ref{thm:irreducible}, the support of $\mathfrak{F}_{\overline{X}}(z)$ will be the sets $S \subset [N]$ of size $d$ that intersect the spread of each $f_i$.  In particular, giving such a collection computes $\mathfrak{F}_{\overline{X}}(z)$ because all coefficients are $0$ or $1$.
If $X$ is irreducible, then the support of $\mathfrak{F}_{\overline{X}}(z)$ is the set of bases of a matroid, and the inclusion-minimal spreads of the $f_i$ are the cocircuits of the matroid  given by the support of $\mathfrak{F}_{\overline{X}}(z)$. The $f_i$ are sometimes called the \emph{circuit polynomials} of $X$, see, e.g., \cite{RST}.

Theorem~\ref{thm:irreducible} allows one to apply powerful tools from intersection theory to verify that a collection of polynomials is a universal Gr\"{o}bner basis. We illustrate the use of Theorem~\ref{thm:irreducible} in Section~\ref{ssec:determinantal} by giving a short proof of the known universal Gr\"{o}bner basis for the ideal of $2 \times 2$ minors. We will apply Theorem~\ref{thm:irreducible} to give a universal Gr\"{o}bner basis for a family of matrix Schubert varieties.

\subsection{Fine Schubert polynomials}

We study the fine multidegree polynomials of matrix Schubert varieties, which we call \textbf{fine Schubert polynomials}. It is known that $X_w$ is a variety of dimension $n^2 - \ell(w)$, where $\ell(w)$ is the length of $w$. Let $\mathfrak{F}_{w}(z) \in \mathbb{Z}[z_{i,j} : (i,j) \in [n]^2]$ denote the fine multidegree polynomial of $X_w \subset \mathbb{A}^{n^2}$. Note that $\mathfrak{F}_{w}(z)$ is  homogeneous of degree $\ell(w)$ and has nonnegative coefficients. Concretely, if we fill in the entries of an $n \times n$ matrix which are labeled by elements of $S^c$ with generic complex numbers, then the coefficient of $z^S$ is the number of ways that the entries in $S$ can be filled in so that the resulting matrix lies in $X_w$. 

\begin{example}
We have $\mathfrak{F}_{2143}(z)=z_{11}(z_{12}+z_{13}+z_{21}+z_{22}+z_{23}+z_{31}+z_{32}+z_{33}). $
\end{example}

We will use fine Schubert polynomials to study the Gr\"{o}bner degenerations of matrix Schubert varieties. Known results on the Gr\"{o}bner theory of matrix Schubert varieties can be translated into information about fine Schubert polynomials. 

\begin{example}
In \cite{KM05}, the authors show that, for certain $\varphi$ (corresponding to an antidiagonal term order), $\lim_{t \to 0} \varphi(t) \cdot X_w$ is the union of the coordinate subspaces corresponding to the \emph{reduced pipe dreams} of \cite{BJS93}. In particular, if $S$ is the set of crosses in a reduced pipe dream for $w$, then the coefficient of $z^S$ in $\mathfrak{F}_{w}(z)$ is positive. 
\end{example}

\begin{example}
In \cite{KW21}, the authors show that there are certain $\varphi$ for which $\lim_{t \to 0} \varphi(t) \cdot X_w$ is a scheme whose reduction is a union of coordinate subspaces corresponding to the blank tiles in the \emph{bumpless pipe dreams} of \cite{LLS}, with the multiplicity of a coordinate subspace being the number of bumpless pipe dreams with the given set of blank tiles.  In particular, the coefficient of $z^S$ in $\mathfrak{F}_{w}(z)$ is at least the number of bumpless pipe dreams with $S$ as the set of blank tiles. Unlike for pipe dreams, there can be multiple bumpless pipe dreams with the same set of blank tiles. 
\end{example}

There are other combinatorial models for Schubert polynomials, such as hybrid pipe dreams \cite{KnutsonUdell}, which conjecturally have Gr\"{o}bner interpretations \cite{KnutsonPersonal}, and so similarly should give lower bounds on the coefficients of fine Schubert polynomials. Fine Schubert polynomials could obstruct the existence of Gr\"{o}bner interpretations for other combinatorial models, such as Bruhat chains \cite{Yu24}. 
\medskip

When $w$ is the permutation in $S_{p + q - r}$ given by $(1, 2, \dotsc, r, q+1, q+2, \dotsc, p + q - r, r + 1, \dotsc, q)$, the matrix Schubert variety $X_w$ is isomorphic to the product of the locus of $p \times q$ matrices of rank at most $r$ with an affine space. In this case, fine Schubert polynomials have been extensively studied in connection with rigidity theory \cite{SingerCucuringu,Bernstein,Tsakiris,BDGGL}. Understanding the support of the fine Schubert polynomials of these permutations is a special case of understanding \emph{bipartite rigidity} \cite{KNN} of graphs.  

When $r \in \{0, 1, p-1, p\}$, these fine Schubert polynomials are easy to understand. When $r =2$, the support of the fine Schubert polynomial was computed in \cite{Bernstein}. When $r \ge p - 3$, the support was computed in \cite{BDGGL}. In both of these cases, the coefficients of $\mathfrak{F}_{w}(z)$ can be more than $1$, and the coefficients are not understood. 

It would be surprising if there is a simple description of the support of this family of fine Schubert polynomials, as the problem of understanding bipartite rigidity of graphs is considered to be of comparable difficulty to the longstanding problem of understanding graph rigidity. 
In \cite[Theorem 1.1]{BDGGL}, it is shown that understanding the support of $\mathfrak{F}_{w}(z)$ in this case is equivalent to a difficult problem in information theory. For this reason, we do not expect there to be a simple recursion for $\mathfrak{F}_{w}(z)$. Nevertheless, we will show that many techniques which are used to compute Schubert polynomials give partial information about fine Schubert polynomials.

\medskip

The \emph{Schubert polynomial} $\mathfrak{S}_w(x) \in \mathbb{Z}[x_1, \dotsc, x_n]$ is a polynomial which represents the class of the Schubert variety corresponding to $w$ in the cohomology of the flag variety. By \cite{KM05}, Schubert polynomials also encode the classes  of matrix Schubert varieties in the $(\mathbb{C}^*)^n$-equivariant cohomology of $\mathbb{A}^{n^2}$. Using this interpretation, we can bound the coefficients of fine Schubert polynomials by the coefficients of Schubert polynomials. For a polynomial $f$ and a monomial $m$, let $\llbracket m \rrbracket f$ denote the coefficient of $m$ in $f$.

\begin{proposition}\label{prop:coeffbound}
Let $w \in S_n$ be a permutation, and let $S$ be a subset of $[n]^2$ of size $\ell(w)$. Then
$$\llbracket z^S\rrbracket \mathfrak{F}_{w}(z) \le \llbracket \prod_{(i, j) \in S}x_i \rrbracket \mathfrak{S}_w(x).$$
\end{proposition}

In particular, if the coefficients of $\mathfrak{S}_w(x)$ are all either $0$ or $1$, then the same is true for $\mathfrak{F}_{w}(z)$. As $X_w$ is isomorphic to $X_{w^{-1}}$ via the map that sends a matrix to its transpose, we have $\mathfrak{F}_{w}(z_{i,j})_{1 \le i, j \le n} = \mathfrak{F}_{w^{-1}}(z_{j,i})_{1 \le j, i \le n}$. 
Therefore, we can also bound the coefficients of $\mathfrak{F}_{w}(z)$ by the coefficients of $\mathfrak{S}_{w^{-1}}(x)$. This gives the following corollary.

\begin{corollary}
Let $w \in S_n$ be a permutation such that the coefficients of either $\mathfrak{S}_w(x)$ or $\mathfrak{S}_{w^{-1}}(x)$ are all either $0$ or $1$. Then the coefficients of $\mathfrak{F}_{w}(z)$ are all either $0$ or $1$. 
\end{corollary}

Permutations $w$ such that the coefficients of $\mathfrak{S}_w(x)$ are all either $0$ or $1$ were classified in \cite{FMS} in terms of a pattern avoidance condition, and this family of permutations has attracted significant attention, see, e.g., \cite{PS22,KleinDiagonal,zeroonegrothendieck,CDGsurvey}.

If all coefficients of a fine Schubert polynomial are either $0$ or $1$, then a result of Brion \cite{BrionMultiplicity} implies that $\overline{X}_w$ has remarkable properties: $\overline{X}_w$ is  projectively normal, arithmetically Cohen--Macaulay, and has rational singularities. Indeed, Brion's result shows that  any integral subvariety $X$ such that the coefficients of $\mathfrak{F}_{\overline{X}}(z)$ are all either $0$ or $1$ will enjoy these properties.

\subsection{Universal Gr\"{o}bner bases for a family of matrix Schubert varieties}

We apply Theorem~\ref{thm:irreducible} to the matrix Schubert varieties $X_w$ of a certain family of permutations. This gives a universal Gr\"{o}bner basis for the ideal of $X_w$ and computes $\mathfrak{F}_{w}(z)$ when $w$ is in this family.

For $1\le p,q \le n-1$, let \[r_{p\times q}(w):=\#\{(i,j)\le (p,q): w(i)=j\}.\]
Let $A$ denote the $n\times n$ matrix whose $(i, j)$th entry is the indeterminate $a_{i,j}$, and let $A_{p\times q}$ be the submatrix of $A$ with rows in $1,\ldots,p$ and columns in $1,\ldots, q$. The ideal $I_w$ of $X_w$ is generated by all minors of size $1+r_{p\times q}(w)$ in $A_{p\times q}$ for all 
\[(p,q)\in \mathrm{ess}(w):=\{(i,j)\in [n-1]^2: w(i)>j, w^{-1}(j)>i, w(i+1)\le j, w^{-1}(j+1)\le i\},\]
see \cite{FultonDegen}. The set $\mathrm{ess}(w)$ is called the \emph{essential set} for $w$, and
 these minors are called the \emph{Fulton generators}. 
Following the terminology in \cite{HPW22}, we also recall the definition of the \emph{CDG generators}. These are obtained from the Fulton generators by setting the monomials that contain variables already in $I_w$ to 0, whenever the Fulton generator is a minor of size at least 2. 
If a generator factors after this specialization, it is redundant and we exclude it
from $\mathrm{CDG}(w)$.
Denote the set of CDG generators for $I_w$ by $\mathrm{CDG}(w)$. It is immediate that the CDG generators also generate $I_w$. 
We now define an iterative procedure that produces a set of generators 
$\mathcal{R}_w$ from the CDG generators of $I_w$. 

Let $f_1,f_2\in \mathbb{Z}[a_{i,j} : (i,j) \in [n]^2]$, and suppose $f_1=a_{\alpha,\beta}g_1+r_1$ and $f_2=a_{\alpha,\beta}g_2+r_2$, where no term in $g_1, g_2, r_1$, or $r_2$ is divisible by $a_{\alpha, \beta}$. Define
\[\mathrm{merge}_{(\alpha,\beta)}(f_1,f_2):=f_1g_2-f_2g_1 = r_1 g_2 - r_2 g_1.\]

The set $\mathcal{R}_w$ is constructed as follows. For $S\subset [n]^2$, let 
$\mathrm{row}(S)$ and $\mathrm{col}(S)$
denote the set of row and column indices
of $S$.

\begin{enumerate}
    \item Initialize $\mathcal{R}_w:=\mathrm{CDG}(w)$;
    \item 
    Let 
    \begin{align*}
        R:=\{&\mathrm{merge}_{(\alpha,\beta)}(f_1,f_2) :  f_1\in \mathcal{R}_w, f_2\in \mathrm{CDG}(w), \, \operatorname{spr}(f_1)\cap \operatorname{spr}(f_2)=\{(\alpha,\beta)\}\\
        &\mathrm{row}(\operatorname{spr}(f_1))\cap \mathrm{row}(\operatorname{spr}(f_2))=\{\alpha\}, \, \mathrm{col}(\operatorname{spr}(f_1))\cap \mathrm{col}(\operatorname{spr}(f_2))=\{\beta\}\}.
    \end{align*}
    Enlarge $\mathcal{R}_w$ by $R$. 
    \item Repeat Step (2) until no new relations are added in this step.
\end{enumerate}
By construction, each element in $\mathcal{R}_w$ is squarefree-supported.

\begin{example}
Let $w=426153.$ There are two relations
in $\mathcal{R}_w$ that are not in
$\mathrm{CDG}(w)$. 
Let 
$f_1=\left| \begin{smallmatrix}
    0 & a_{1,4} & a_{1,5} \\
    a_{2,2} & a_{2,4} & a_{2,5}\\
    a_{3,2} & a_{3,4} & a_{3,5}
\end{smallmatrix}\right|$, 
$f_2=\left| \begin{smallmatrix}
    0 & a_{1,4} & a_{1,5} \\
    a_{2,3} & a_{2,4} & a_{2,5}\\
    a_{3,3} & a_{3,4} & a_{3,5}
\end{smallmatrix}\right|$,
$f_3=\left| \begin{smallmatrix}
    0 & a_{2,2} & a_{2,3} \\
    a_{4,1} & a_{4,2} & a_{4,3}\\
    a_{5,1} & a_{5,2} & a_{5,3}
\end{smallmatrix}\right|$, and
$f_4=\left| \begin{smallmatrix}
    0 & a_{3,2} & a_{3,3} \\
    a_{4,1} & a_{4,2} & a_{4,3}\\
    a_{5,1} & a_{5,2} & a_{5,3}
\end{smallmatrix}\right|$.
Then $h_1:=\mathrm{merge}_{(3,2)}(f_1,f_4)=-\mathrm{merge}_{(2,3)}(f_2,f_3)$
and $h_2:=\mathrm{merge}_{(2,2)}(f_1,f_3)=-\mathrm{merge}_{(3,3)}(f_2,f_4)$
are the two extra relations.
The sets $\operatorname{spr}(h_1)$
and $\operatorname{spr}(h_2)$ are shown
in Figure~\ref{fig:426153}. 
\begin{figure}[h]
    \centering
    \includegraphics[width=0.5\linewidth]{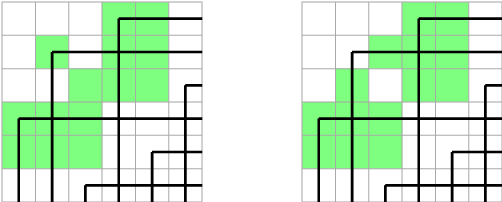}
    \caption{Spreads of generators in $\mathcal{R}_{w}$ not in $\mathrm{CDG}(w)$ for $w=426153$.}
    \label{fig:426153}
\end{figure}
\end{example}

One can show that, in general, the elements $f \in \mathcal{R}_w$ can be described as determinants of matrices whose entries are $0$ except on the spread of $f$. 

\begin{theorem}\label{thm:main}
Let $w$ be a permutation such that all coefficients of both $\mathfrak{S}_w(x)$ and $\mathfrak{S}_{w^{-1}}(x)$ are either $0$ or $1$. 
Then $\Delta(f : f \in \mathcal{R}_w)$ is pure of dimension $n^2 - \ell(w) - 1$, and for each facet $U$ of $\Delta(f : f \in \mathcal{R}_w)$, $\llbracket z^{U^c}\rrbracket \mathfrak{F}_{w}(z) = 1$. 
In particular, $\mathcal{R}_w$ is a universal Gr\"{o}bner basis for $I_w$. 
\end{theorem}

The second part of Theorem~\ref{thm:main} follows from the first part and Theorem~\ref{thm:irreducible}.
If $U$ is not a facet of $\Delta(f : f \in \mathcal{R}_w)$, then $\llbracket z^{U^c}\rrbracket \mathfrak{F}_{w}(z) = 0$, so Theorem~\ref{thm:main} computes $\mathfrak{F}_{w}(z)$. Explicitly, we have
$$\mathfrak{F}_{w}(z) = \sum_{\substack{|S| = \ell(w),  \\ S \cap \operatorname{spr}(f)\neq\emptyset \text{ for all }f \in \mathcal{R}_w}} \prod_{(i,j) \in S} z_{i,j}.$$

\begin{example}
    For $w=426153$, we describe the terms in $\mathfrak{F}_w(z)$. Let $A=\{1,2,3\}\times \{4,5\}$, $B=\{2,3\}\times \{2,3\}$, and 
    $C=\{4,5\}\times\{1,2,3\}$. Then $\llbracket z^S\rrbracket \mathfrak{F}_w$ is nonzero if and only if $S$ contains
    $\{(1,1), (1,2), (1,3), (2,1), (3,1)\}$ as well as three more elements,  chosen in one of the four ways: (1) one element each from $A$, $B$, and $C$;
    (2) one element from $A$, and two elements from $B$ that 
    are not in the same row;
    (3) one element from $C$, and 
    two elements from $B$ that
    do not use the same column;
    or
    (4) three elements from $B$. In all cases,
    $\llbracket z^S\rrbracket \mathfrak{F}_w=1$.
\end{example}
A permutation $w$ for which the matrix Schubert variety corresponds to the variety of $p \times q$ matrices of rank at most $r$, with $p \le q$, satisfies the condition in Theorem~\ref{thm:main} if and only if $r \in \{0, 1, p-1, p\}$. 
Theorem~\ref{thm:main} generalizes the known universal Gr\"{o}bner bases for the ideals of maximal minors and the ideals of $2 \times 2$ minors. Prior to our work, universal Gr\"{o}bner bases for the ideals of matrix Schubert varieties were known in only a few cases beyond this, such as matrix Schubert varieties which are, up to affine factors, products of the above cases
\cite[Proposition 3.2]{HPW22}.

We give an explicit description in terms of pattern avoidance of the permutations $w$ for which all coefficients of both $\mathfrak{S}_w(x)$ and $\mathfrak{S}_{w^{-1}}(x)$ are either $0$ or $1$, see Proposition~\ref{prop:patterncriterion}. 
We also give a positive description of these permutations in terms of their Rothe diagrams, see Theorem~\ref{thm:characterization}. See Figure~\ref{fig:largeperm} for an example
of a large permutation in this class.

\begin{remark}\label{rem:characteristic}
Matrix Schubert varieties can be defined over any field, and one can formulate the definition of the fine Schubert polynomial over any field (using Chow groups instead of homology). One can check that the multidegree of the closure of a matrix Schubert variety in $(\mathbb{P}^1)^{n^2}$ depends only on the characteristic of the field, but it is not obvious that it is independent of the field. See \cite[Section 2.3]{BDGGL} for a discussion in the case of determinantal varieties. The proofs of all statements in this paper work over any field. In particular, for any $w$ such that all coefficients of both $\mathfrak{S}_w(x)$ and $\mathfrak{S}_{w^{-1}}(x)$ are equal to $0$ or $1$, the multidegree of the closure of $X_w$ in $(\mathbb{P}^1)^{n^2}$ is independent of the characteristic, and $\mathcal{R}_w$ is a universal Gr\"{o}bner basis. 
\end{remark}

For arbitrary $w$, the relations in $\mathcal{R}_w$ are not a universal Gr\"{o}bner basis, and they are not enough to compute the support of $\mathfrak{F}_{w}(z)$. See Example~\ref{ex:21543}.

Unlike the applications of Theorem~\ref{thm:irreducible} to recover existing universal Gr\"{o}bner basis results, the proof of Theorem~\ref{thm:main} relies on a complicated induction. This is because we do not have an explicit description of the facets of $\Delta(f : f \in \mathcal{R}_w)$. 
We instead develop several tools to relate both $\mathfrak{F}_{w}(z)$ and the faces of $\Delta(f : f \in \mathcal{R}_w)$ to fine Schubert polynomials of ``smaller'' permutations. We show that these tools, taken together, are enough to understand all of the facets of $\Delta(f : f \in \mathcal{R}_w)$.  These tools include a version of Lascoux's transition formula \cite{Lascoux} (Proposition~\ref{prop:transition}) and Knutson's cotransition formula \cite{KnutsonCotransition} (Proposition~\ref{prop:cotransition}). 
We highlight one tool, which shows that fine Schubert polynomials behave well with respect to pattern containment. For $w\in S_n$, let
\[D(w):=\{(i,j)\in [n]^2: w(i)>j, w^{-1}(j)>i\}\]
be the \textbf{Rothe diagram} of $w$. Define a \textbf{region} of $D(w)$ to be a maximal subset of $D(w)$ connected via adjacency. Note that
for all $(p,q)$ in
a region, $r_{p\times q}(w)$ is the same. We call this the \textbf{rank} of this region.

\begin{proposition}\label{prop:patterncontainment}
Let $w \in S_n$ be a permutation. For some $k \in [n]$, let $v \in S_{n-1}$ be the permutation given by the relative ordering of $w(1), \dotsc, w(k-1), w(k+1), \dotsc, w(n)$. 
Set $M = \prod z_{i,j}$, where the product is over pairs $(i, j)$ with $(i, j) \in D(w)$ and $i = k$ or $j = w(k)$. Then
$$\mathfrak{F}_{w}(z) = M \cdot \mathfrak{F}_{v}(z_{i,j})_{i \in [n] \setminus \{k\}, j \in [n] \setminus \{w(k)\}} + E,$$
where $E$ has non-negative coefficients. 
\end{proposition}

In particular, the family of $w \in S_n$ where $\mathfrak{F}_{w}(z)$ has all coefficients equal to either  $0$ or $1$ is closed under pattern containment. Proposition~\ref{prop:patterncontainment} is inspired by \cite[Theorem 1.2]{FMS}, which gives a version of Proposition~\ref{prop:patterncontainment} for Schubert polynomials. Our proof of Proposition~\ref{prop:patterncontainment} generalizes to give a simple geometric proof of \cite[Theorem 1.2]{FMS}, see Remark~\ref{rem:schubertpattern}.

Our paper is organized as follows. In Section~\ref{sec:grobner}, we study the relationship between Gr\"{o}bner bases and closures of varieties in a product of projective lines, and we prove Theorem~\ref{thm:irreducible}. In Section~\ref{sec:geom}, we develop several inductive tools to study fine Schubert polynomials and prove Propositions~\ref{prop:coeffbound} and \ref{prop:patterncontainment}. In Section~\ref{sec:patterns}, we analyze the permutations $w$ with all coefficients of both $\mathfrak{S}_w(x)$ and $\mathfrak{S}_{w^{-1}}(x)$ equal to $0$ or $1$. In Section~\ref{sec:proof}, we prove Theorem~\ref{thm:main}. 
\medskip

\noindent
\textbf{Acknowledgments}:
We thank Aldo Conca, Patricia Klein, Allen Knutson, Karen Smith, Bernd Sturmfels, and Josephine Yu for helpful conversations, and we thank Dave Anderson, Sara Billey, and William Fulton for helpful comments on an earlier version of this paper. We thank Elizabeth Pratt for giving us the example in Remark~\ref{rem:lex}. We thank the referees for their close reading and useful suggestions. 
Macaulay2 \cite{M2} was invaluable for computations throughout this project, and we thank Christopher Eur for providing code to compute multihomogenizations of ideals. ChatGPT 5.5 Pro was used to proofread this paper, but the content of the paper was created by the authors. 
This project began during the ICERM ``Spring 2021 Reunion Event.''
We thank Melody Chan for her excellent organization of that event. DH was supported by NSF-DMS2202900 and the Charles Simonyi Endowment at the Institute for Advanced Study. 

\section{Fine multidegrees and Gr\"{o}bner bases}\label{sec:grobner}

In this section, we prove some results describing how the ideal of an affine variety relates to the homology classes of the closure of that variety in a product of projective lines. We then prove Theorem~\ref{thm:irreducible}. We illustrate some applications of Theorem~\ref{thm:irreducible} in Section~\ref{ssec:determinantal}. 

\subsection{Fine multidegrees}

For a closed subscheme $Y$ of $(\mathbb{P}^1)^N$ whose irreducible components all have codimension $d$, let $\mathfrak{F}_Y(z) \in \mathbb{Z}[z_1, \dotsc, z_N]$ be the unique squarefree-supported polynomial such that 
$$[Y] = \mathfrak{F}_Y([H_1], \dotsc, [H_N]) \smallfrown [(\mathbb{P}^1)^N].$$
Then $\mathfrak{F}_Y(z)$ is homogeneous of degree $d$ and has nonnegative coefficients.
We call $\mathfrak{F}_Y(z)$ the \textbf{fine multidegree polynomial} of $Y$.
If $Y$ is irreducible, then this is the covolume polynomial of $Y$ in the sense of \cite{Aluffi}, and it is a typical example of a dually Lorentzian polynomial in the sense of \cite{RSW}. 

Let $\mathbb{C}[a_1, \dotsc, a_N, b_1, \dotsc, b_N]$ be the coordinate ring of $(\mathbb{P}^1)^N$, which has a $\mathbb{Z}^{N}$ grading, with the degree of $a_i$ and $b_i$ equal to the $i$th standard basis vector $\mathbf{e}_i$. Then the ideal $\mathcal{I}_Y$ of $Y$ in $(\mathbb{P}^1)^N$ is also $\mathbb{Z}^N$-graded. When $Y$ is pure dimensional, the fine multidegree polynomial can be computed in terms of the fine Hilbert series of $\mathcal{I}_Y$, i.e., the $\mathbb{Z}^N$-graded Hilbert series of $\mathcal{I}_Y$, see, e.g., \cite[Section 8.5]{MillerSturmfels}.

Let $X \subset \mathbb{A}^N$ be a closed subscheme, and let $\overline{X}$ be the closure of $X$ in $(\mathbb{P}^1)^N$. 
Let $\mathbb{C}[a_1, \dotsc, a_N]$ be the coordinate ring of $\mathbb{A}^N$, and let ${I}_X \subset \mathbb{C}[a_1, \dotsc, a_N]$ be the ideal of $X$. For $f \in \mathbb{C}[a_1, \dotsc, a_N]$, let $f^h \in \mathbb{C}[a_1, \dotsc, a_N, b_1, \dotsc, b_N]$ be the multihomogenization of $f$, i.e., $f^h$ is the multihomogeneous polynomial of minimal degree such that $f^h(a_1, \dotsc, a_N, 1, \dotsc, 1) = f(a_1, \dotsc, a_N)$.  Let
$$\mathcal{I}_X^h := (f^h : f \in {I}_X).$$ 
Note that $\mathcal{I}_X^h = \mathcal{I}_{\overline{X}}$, the ideal of $\overline{X}$ in $\mathbb{C}[a_1, \dotsc, a_N, b_1, \dotsc, b_N]$. 

For a nonzero polynomial $f \in \mathbb{C}[a_1, \dotsc, a_N]$, the \textbf{spread} of $f$ is the subset
$$\operatorname{spr}(f) = \{i : a_i \text{ appears in a monomial with nonzero coefficient in } f\} \subset [N].$$
I.e., the spread of $f$ is the variables which it uses. Similarly, for a nonzero $f \in \mathbb{C}[a_1, \dotsc, a_N, b_1, \dotsc, b_N]$, we set
$$\operatorname{spr}(f) = \{i : a_i \text{ or }b_i \text{ appears in a monomial with nonzero coefficient in } f\} \subset [N].$$
Note that $\operatorname{spr}(f) = \operatorname{spr}(f^h)$, and if $g \in \mathbb{C}[a_1, \dotsc, a_N, b_1, \dotsc, b_N]$ is multihomogeneous of degree $(d_1, \dotsc, d_N)$, then $\operatorname{spr}(g) = \{i : d_i \not= 0\}$. 

\medskip

We now prove a few results which relate the fine multidegree of $\overline{X}$ to the spreads of elements of $I_X$. 
For $U \subset [N]$, let $\pi_U \colon (\mathbb{P}^1)^N \to (\mathbb{P}^1)^U$ be the coordinate projection. 

\begin{lemma}\label{lem:spreadprojection}
Let $Y$ be a closed subscheme of $(\mathbb{P}^1)^N$.
For $U \subset [N]$, we have $\pi_U(Y) \not= (\mathbb{P}^1)^U$ if and only if there is  a nonzero $f \in \mathcal{I}_Y$ with $\operatorname{spr}(f) \subset U$. 
\end{lemma}

\begin{proof}
Because $\pi_U$ is a closed map, $\pi_U(Y)$ is a Zariski closed subset of $(\mathbb{P}^1)^U$, so it is a proper subset if and only if it is contained in some divisor, i.e., if there is a nonzero multihomogeneous polynomial $g \in \mathbb{C}[a_i, b_i : i \in U]$ such that $\pi_U(Y) \subset V(g)$. Then $g$ defines an element of $\mathcal{I}_Y$ with spread contained in $U$, via the inclusion of $\mathbb{C}[a_i, b_i : i \in U]$ into $\mathbb{C}[a_1, \dotsc, a_N, b_1, \dotsc, b_N]$. 

For the reverse direction, if $\operatorname{spr}(f) \subset U$ for some nonzero $f \in \mathcal{I}_Y$, then $f$, viewed as an element of $\mathbb{C}[a_i, b_i : i \in U]$, will vanish on $\pi_U(Y)$, so $\pi_U(Y)$ must be properly contained in $(\mathbb{P}^1)^U$. 
\end{proof}

\begin{lemma}\label{lem:dehomogenize}
There is a nonzero $g \in \mathcal{I}_X^h$ with $\operatorname{spr}(g) = C$ if and only if there is a nonzero $f \in {I}_X$ with $\operatorname{spr}(f) = C$. 
\end{lemma}

\begin{proof}
One direction follows from the fact that $\operatorname{spr}(f) = \operatorname{spr}(f^h)$. 

For the other direction, note that if $S$ is the spread of some $0 \not= g \in \mathcal{I}_X^h$ and $T \supset S$, then $T = \operatorname{spr}(a^{T \setminus S} \cdot g)$. It therefore suffices to show that the minimal elements of $\{\operatorname{spr}(g) : 0 \not= g \in \mathcal{I}_X^h\}$ are the spreads of elements in $I_X$. 

As the spread of $g \in \mathcal{I}_X^h$ is the union of the spreads of its multihomogeneous parts, we may assume that $g$ is multihomogeneous. Then $\operatorname{spr}(g(a_1, \dotsc, a_N, 1, \dotsc, 1))$ is a subset of $\operatorname{spr}(g)$, and we can obtain an element of $I_X$ with the correct spread by multiplying by an appropriate monomial. 
\end{proof}

The following lemma is closely related to the fact that, when $Y$ is irreducible, the support of the fine multidegree polynomial of $Y$ is the bases of the dual of the algebraic matroid of $Y$. 

\begin{lemma}\label{lem:algebraicmatroid}
Suppose all irreducible components of $Y$ have codimension $d$. If $\pi_U(Y) = (\mathbb{P}^1)^U$, then there is $V \supset U$ with $|V| = N - d$ and $\pi_V(Y) = (\mathbb{P}^1)^V$. 
\end{lemma}

\begin{proof}
By replacing $Y$ with the reduction of a component which dominates $(\mathbb{P}^1)^U$, we can assume that $Y$ is integral. It suffices to show that if $|U| < N - d$, then we can find $W \supset U$ with $|W| = |U| + 1$ and $\pi_W(Y) = (\mathbb{P}^1)^W$. 

For each $k \in [N] \setminus U$, if $\pi_{U \cup \{k\}}(Y) \not = (\mathbb{P}^1)^{U \cup \{k\}}$, then by Lemma~\ref{lem:spreadprojection} there is a nonzero $f_k \in \mathcal{I}_Y$ whose spread is contained in $U \cup \{k\}$.  As $\pi_U(Y) = (\mathbb{P}^1)^U$, the spread of $f_k$ must contain $k$. Inside of the function field of $Y$, $f_k$ gives an algebraic relation satisfied by $a_k$ over the field generated by $\{a_\ell : \ell \in U\}$. 
If this holds for all $k \in [N] \setminus U$, this implies that the function field of $Y$ (which is generated by $\{a_1, \dotsc, a_N\}$) is algebraic over the subfield generated by $\{a_\ell : \ell \in U\}$. But this contradicts that $|U| < N - d$. 
\end{proof}

\begin{lemma}\label{lem:proj}
Suppose all irreducible components of $X$ have codimension $d$. For any $U \subset [N]$ with $|U| = N - d$, $\pi_U(\overline{X}) = (\mathbb{P}^1)^U$ if and only if $\llbracket z^{U^c} \rrbracket \mathfrak{F}_{\overline{X}}(z) > 0$.
\end{lemma}

\begin{proof}
As $\dim \overline{X} = \dim \, (\mathbb{P}^1)^U$, the map $\overline{X} \to (\mathbb{P}^1)^U$ is surjective if and only if the pushforward of $[\overline{X}] \in H_{2(N - d)}((\mathbb{P}^1)^N)$ to $H_{2(N - d)}((\mathbb{P}^1)^U)$ is nonzero. Under the natural identification of $H_{2(N - d)}((\mathbb{P}^1)^U)$ with $\mathbb{Z}$, the pushforward of $[\overline{X}]$ is equal to $\llbracket z^{U^c} \rrbracket \mathfrak{F}_{\overline{X}}(z)$. 
\end{proof}

\begin{proposition}\label{prop:spreadmultidegree}
If all irreducible components of $Y$ have codimension $d$, then we can describe the spreads of nonzero elements of $\mathcal{I}_Y$, as follows. For $S \subset [N]$,
$$\text{there is }0 \not= f \in \mathcal{I}_Y \text{ with } \operatorname{spr}(f) \cap S = \emptyset \text{ if and only if } \llbracket z^T\rrbracket \mathfrak{F}_Y(z) = 0 \text{ for all }T \subset S \text{ with }|T| = d.$$
\end{proposition}

\begin{proof}
Lemma~\ref{lem:spreadprojection} implies that if $\operatorname{spr}(f) \cap S = \emptyset$ for some nonzero $f \in \mathcal{I}_Y$, then $\llbracket z^T\rrbracket \mathfrak{F}_Y(z) = 0 \text{ for all }T \subset S$ because $\pi_{T^c}(Y)$ is a proper subset of $(\mathbb{P}^1)^{T^c}$. The other direction follows from Lemma~\ref{lem:algebraicmatroid} and Lemma~\ref{lem:proj}.
\end{proof}

Recall that for nonzero polynomials $f_1, \dotsc, f_r$, $\Delta(f_1, \dotsc, f_r)$ is the simplicial complex on $[N]$ whose nonfaces are generated by the spreads of the $f_i$.

\begin{proposition}\label{prop:pure}
Let $f_1, \dotsc, f_r$ be a Gr\"{o}bner basis for $I_X$ with respect to every lexicographic term order. Assume that all $f_i$ are nonzero. Then $\Delta(f_1, \dotsc, f_r) = \{U : \pi_U(\overline{X}) = (\mathbb{P}^1)^U\}$. If all irreducible components of $X$ have codimension $d$, then $\Delta(f_1, \dotsc, f_r)$ is pure of dimension $N - d - 1$, and $U \subset [N]$ is a facet if and only if $\llbracket z^{U^c} \rrbracket \mathfrak{F}_{\overline{X}}(z) > 0$. 
\end{proposition}

\begin{proof}
Let $\Delta = \{U : \pi_U(\overline{X}) = (\mathbb{P}^1)^U\}$. 
It follows from Lemma~\ref{lem:spreadprojection} and Lemma~\ref{lem:dehomogenize} that $\pi_U(\overline{X}) = (\mathbb{P}^1)^U$ if and only if there is no nonzero $g \in I_X$ with $\operatorname{spr}(g) \subset U$, so $\Delta \subset \Delta(f_1, \dotsc, f_r)$. 

Let $U$ be a nonface of $\Delta$, and let $<$ be a lexicographic term order on $\mathbb{C}[a_1, \dotsc, a_N]$ in which the elements of $U$ occur last. Because $U$ is a nonface, there is some nonzero $g \in I_X$ such that $\operatorname{spr}(g) \subset U$. Then $\operatorname{spr}(\operatorname{in}_<(g)) \subset U$, so there is some $f_i$ such that $\operatorname{spr}(\operatorname{in}_{<}(f_i)) \subset U$ because $f_1, \dotsc, f_r$ is a Gr\"{o}bner basis with respect to $<$. But the choice of term order then implies that $\operatorname{spr}(f_i) \subset U$, so $U$ is a nonface of $\Delta(f_1, \dotsc, f_r)$. 

Now suppose that all irreducible components of $X$ have codimension $d$. Then $\overline{X}$ cannot surject onto any $(\mathbb{P}^1)^U$ with $|U| > N - d$, so the dimension of $\Delta$ is at most $N - d - 1$. 
Lemma~\ref{lem:algebraicmatroid} then implies that $\Delta$ is pure of dimension $N - d - 1$, and Lemma~\ref{lem:proj} shows the characterization of the facets of $\Delta$ in terms of the fine multidegree. 
\end{proof}

\subsection{A Gr\"{o}bner basis criterion}
We now prove a generalization of Theorem~\ref{thm:irreducible} which does not assume that all irreducible components of $X$ have the same dimension. 

\begin{theorem}\label{thm:grobnercriterion}
Let $X$ be a closed subscheme of $\mathbb{A}^N$, and let $f_1, \dotsc, f_r \in {I}_X$ be nonzero squarefree-supported polynomials. Suppose that, for all facets $U$ of $\Delta(f_1, \dotsc, f_r)$, we have $\pi_{U}(\overline{X}) = (\mathbb{P}^1)^{U}.$
Then $\{f_1, \dotsc, f_r\}$ is a universal Gr\"{o}bner basis for ${I}_X$. If $X$ is pure dimensional, then all coefficients of $\mathfrak{F}_{\overline{X}}(z)$ are $0$ or $1$. 
\end{theorem}

In particular, when the hypothesis of Theorem~\ref{thm:grobnercriterion} holds, all initial ideals of $X$ are reduced. The proof of Theorem~\ref{thm:grobnercriterion} shows that $\{f_1^h, \dotsc, f_r^h\}$ is a universal Gr\"{o}bner basis for $\mathcal{I}_X^h$.

\begin{proof}[Proof of Theorem~\ref{thm:irreducible}]
By Theorem~\ref{thm:grobnercriterion}, it suffices to check that if all irreducible components of $X$ have codimension $d$ and that $|U| = N - d$, then $\pi_U(\overline{X}) = (\mathbb{P}^1)^U$ if and only if $\llbracket z^{U^c} \rrbracket \mathfrak{F}_{\overline{X}}(z) > 0$. This is Lemma~\ref{lem:proj}.
\end{proof}

Note that Lemma~\ref{lem:algebraicmatroid} implies that if all irreducible components of $X$ have codimension $d$ and the hypothesis of Theorem~\ref{thm:grobnercriterion} is satisfied, then $\Delta(f_1, \dotsc, f_r)$ is pure of dimension $N - d - 1$.

In the proof of Theorem~\ref{thm:grobnercriterion}, we will need a combinatorial result. Let $\Delta$ be a simplicial complex on $[N]$, and let $[N, \bar{N}] = [1, \dotsc, N, \bar{1}, \dots, \bar{N}]$ be a set of size $2N$ equipped with the obvious involution $\bar{\cdot}$. Let $q \colon [N, \bar{N}] \to [N]$ be the quotient by $\bar{\cdot}$. 
For a minimal  nonface $G$ of $\Delta$, a \textbf{lift} $\widetilde{G}$ to $[N, \bar{N}]$ is a subset of $[N, \bar{N}]$ that contains exactly one of $\{i, \bar{i}\}$ for each $i \in G$ and has $q(\widetilde{G}) = G$. 

Suppose we have chosen a lift $\widetilde{G}$ for each minimal nonface $G$ of $\Delta$. Let $\widetilde{\Delta}$ be the complex on $[N, \bar{N}]$ whose minimal nonfaces are given by the $\widetilde{G}$. For each face $F$ of $\widetilde{\Delta}$, $\{i : \{i, \bar{i}\} \subset F\}$ is a face of $\Delta$. We say that the choice of lifts is \textbf{full} if, for each facet $F$ of $\Delta$, there is a face $\widetilde{F}$ of $\widetilde{\Delta}$ such that $\widetilde{F} \cap \{i, \bar{i}\} \not= \emptyset$ for all $i$, and $F = \{i : \{i, \bar{i}\} \subset \widetilde{F}\}$. We say that $\widetilde{F}$ \textbf{extends} $F$. 

\begin{lemma}\label{lem:tildelift}
Suppose that $\widetilde{\Delta}$ arises from a full choice of lifts. Then every facet of $\Delta$ has a unique facet of $\widetilde{\Delta}$ extending it, and all facets of $\widetilde{\Delta}$ extend a facet of $\Delta$. 
\end{lemma}

\begin{proof}
We induct on $N$; the statement has no content for $N = 0$. 
Let $D:=\{E\in \Delta: N\not\in E\}$ be
the deletion of $N$ and
$L:=\{E\in D:E\cup \{N\}\in \Delta\}$
be the link of $N$.

Each minimal nonface of $\Delta$ gives rise to a nonface of $L$, with a minimal nonface $G$ which contains $N$ giving rise to the nonface $G \setminus \{N\}$ of $L$, and a minimal nonface $H$ which does not contain $N$ giving rise to the nonface $H$ of $L$. Every minimal nonface of $L$ arises in this way, and these nonfaces all have induced choices of lifts.  Let $F$ be a facet of $\Delta$ which contains $N$. Because the choice of lifts is full, there is a facet $\widetilde{F}$ of $\widetilde{\Delta}$ extending $F$. Then $\widetilde{F} \setminus \{N, \bar{N}\}$ extends the facet $F \setminus \{N\}$ of $L$, so we deduce that this induced choice of lifts for $L$ is full. Let $\widetilde{L}$ be the corresponding complex, which is the link of $\{N, \bar{N}\}$ in $\widetilde{\Delta}$.

The faces of $L$ are of the form $F \setminus \{N\}$, where $F$ is a face of $\Delta$ containing $N$. If we take a face of $\widetilde{L}$, we can obtain faces of $\widetilde{\Delta}$ by adding any (possibly empty) subset of $\{N, \bar{N}\}$. 

Each minimal nonface $G$ of $\Delta$ which does not contain $N$ gives rise to a minimal nonface $G$ of $D$, and every minimal nonface of $D$ arises in this way. We have an induced choice of lifts for $D$. Because the choice of lifts for $\Delta$ is full, the choice of lifts for $D$ is full. Let $\widetilde{D}$ be the corresponding complex, which is $\widetilde{\Delta} \setminus \{N, \bar{N}\}$. 

There are two types of faces of $D$: faces $F$ such that $F \cup \{N\}$ is a face of $\Delta$, and faces $K$ such that $K \cup \{N\}$ is not a face of $\Delta$. Such faces $F$ also correspond to faces of $L$. We can obtain a face of $\widetilde{\Delta}$ by taking a face of the first kind in $\widetilde{D}$ and adding a subset of $\{N, \bar{N}\}$; we get the same faces from $\widetilde{L}$. 

Let $\widetilde{K}$ be a face of $\widetilde{D}$ which corresponds to the second kind of face of $D$. We know that $\widetilde{K}$ is a face of $\widetilde{\Delta}$ and that $\widetilde{K} \cup \{N, \bar{N}\}$ is not a face of $\widetilde{\Delta}$. Also, at most one of $\widetilde{K} \cup \{N\}$ and $\widetilde{K} \cup \{\bar{N}\}$ is a face of $\widetilde{\Delta}$ (because $K \cup \{N\}$ contains at least one minimal nonface of $\Delta$). However, as the lift of $\widetilde{\Delta}$ is full, for each facet of the second kind in $D$, we can extend it to a facet of $\widetilde{\Delta}$ that contains $N$ or $\bar{N}$. So we can in fact extend $\widetilde{K}$ in exactly one way. 

This describes all of the faces of $\widetilde{\Delta}$. 
The claim follows, as this description of the facets of $\widetilde{\Delta}$ is in bijection with the facets of $\Delta$. 
\end{proof}

\begin{example}
Suppose that $N=5$ and $\Delta$ has facets $\{1,4\}, \{1,5\}, \{2,5\}, \{3,5\}$. If we choose the lifts
$$\{\bar{4}, \bar{5}\}, \{\bar{3}, \bar{4}\}, \{\bar{2}, \bar{3}\}, \{\bar{1}, 2\}, \{\bar{1}, \bar{3}\}, \{2, \bar{4}\}$$
of the minimal nonfaces of $\Delta$, then the facets of $\widetilde{\Delta}$ are 
$$\{1, \bar{1}, \bar{2}, 3, 4, \bar{4}, 5\}, \{1, \bar{1}, \bar{2}, 3, 4, 5, \bar{5}\}, \{1, 2, \bar{2}, 3, 4, 5, \bar{5}\}, \{1, 2, 3, \bar{3}, 4, 5, \bar{5}\},$$
and so this choice of lifts is full. 
\end{example}

\begin{example}
The external activity complex of a matroid \cite{ArdilaBoocher} is an example of a full choice of lifts $\widetilde{\Delta}$, where $\Delta$ is the independence complex of the matroid. Lemma~\ref{lem:tildelift} can be used to give a quick proof of \cite[Theorem 1.5]{ArdilaBoocher}.  More generally, the external activity complex of a pair of matroids \cite{BergetFink} is an example of a full choice of lifts of the independence complex of the diagonal Dilworth truncation of a pair of matroids \cite[Proposition 7.3]{BergetFink}. It does not seem easy to directly verify the hypothesis of Lemma~\ref{lem:tildelift} in this case.  
\end{example}

\medskip

By the Stanley--Reisner correspondence, the irreducible components of the vanishing locus in affine space of a squarefree monomial ideal are in bijection with the facets of the simplicial complex whose nonfaces are the squarefree monomials in the ideal, see, e.g., \cite[Theorem 1.7]{MillerSturmfels}. We will need a variation of this for vanishing loci of squarefree monomial ideals in $(\mathbb{P}^1)^N$.

Let $\widetilde{\Delta}$ be a simplicial complex on $[N, \bar{N}]$. For each minimal nonface $F$ of $\widetilde{\Delta}$, let $m_F = (\prod_{i \in F} a_i) (\prod_{\bar{j} \in F} b_j)$, and let $\mathcal{I} = (m_F : F \text{ minimal nonface of }\widetilde{\Delta})$. Then the irreducible components of the vanishing locus of $\mathcal{I}$ in $\mathbb{A}^{[N, \bar{N}]}$ are in bijection with facets of $\widetilde{\Delta}$, with the locus $\bigcap_{\ell\in [N]\setminus T}V(a_\ell)\cap \bigcap_{\bar{p}\in [\bar{N}]\setminus T}V(b_p)$ corresponding to the facet $T$.

We obtain the vanishing locus of $\mathcal{I}$ in $(\mathbb{P}^1)^N$ by taking the vanishing locus in $\mathbb{A}^{[N, \bar{N}]}$, intersecting with the open set $U = \{a_i \not=0 \text{ or }b_i \not=0 \text{ for each }i\}$, and then taking the image under the map $U \to (\mathbb{P}^1)^N$ given by quotienting by the natural $\mathbb{G}_m^N$ action on $U$. 

We see that the components of $V(\mathcal{I})$ in $(\mathbb{P}^1)^N$ are in bijection with facets $T$ of $\widetilde{\Delta}$ such that $\{i, \bar{i}\} \cap T \not= \emptyset$ for all $i$; with such $T$ corresponding to the locus $V_T$ in $(\mathbb{P}^1)^N$ where the $i$th coordinate is allowed to be arbitrary if $\{i, \bar{i}\} \subset T$, the $i$th coordinate is fixed to be $\infty$ if $i \in T$ but $\bar{i} \not \in T$, and the $i$th coordinate is fixed to be $0$ if $\bar{i} \in T$ but $i \not \in T$. 

The facets $T$ of $\widetilde{\Delta}$ which have $T \cap \{i, \bar{i}\} = \emptyset$ for some $i$ do not affect $V(\mathcal{I})$: they correspond to associated primes of $\mathcal{I}$ which contain the irrelevant ideal.

\begin{proof}[Proof of Theorem~\ref{thm:grobnercriterion}]
Let $<$ be a term order on $\mathbb{C}[a_1, \dotsc, a_N]$. We can extend $<$ to a term order on $\mathbb{C}[a_1, \dotsc, a_N, b_1, \dotsc, b_N]$ such that $a^Sb^T < a^{S'}b^{T'}$ if $a^S < a^{S'}$; the initial term of a multihomogeneous polynomial in $\mathbb{C}[a_1, \dotsc, a_N, b_1, \dotsc, b_N]$ does not depend on the choice of extension. We have an inclusion
$$V(\operatorname{in}_{<}(\mathcal{I}_X^h)) \subset V(\operatorname{in}_{<}(f_1^h), \dotsc, \operatorname{in}_{<}(f_r^h))\subset (\mathbb{P}^1)^N.$$
Let $\Delta$ be the simplicial complex on $[N]$ whose faces are given by subsets $S$ of $[N]$ such that $\pi_S(\overline{X}) = (\mathbb{P}^1)^S$. Every minimal nonface of $\Delta$ is the spread of some $f_i$. Indeed, if $F$ is a minimal nonface of $\Delta$, by assumption, there exists some $f_i$ such that $\operatorname{spr}(f_i)\subset F$. By Lemma~\ref{lem:spreadprojection} together with Lemma~\ref{lem:dehomogenize}, $\operatorname{spr}(f_i)$ is a nonface, so $\operatorname{spr}(f_i)=F$. By removing some elements from $\{f_1, \dotsc, f_r\}$, we can assume that each minimal nonface of $\Delta$ is the spread of a unique $f_i$. 

Because each $f_i$ is squarefree-supported, $\operatorname{in}_{<}(f_i^h)$ is a squarefree monomial where, for each $j$, at most one of $\{a_j, b_j\}$ appears. Additionally, we have $\operatorname{spr}(\operatorname{in}_{<}(f_i^h)) = \operatorname{spr}(f_i)$. We identify this squarefree monomial with a subset of $[N, \bar{N}]$, where the $b$-variables correspond to $[\bar{N}]$. Then this is a lift of the nonface $\operatorname{spr}(f_i)$ of $\Delta$. 
Let $\widetilde{\Delta}$ be the complex on $[N, \bar{N}]$ with these as nonfaces.
Note that $\widetilde{\Delta}$ arises from a choice of lifts of the minimal nonfaces of $\Delta$ (in the sense discussed above Lemma~\ref{lem:tildelift}).

As described above, each facet $T$ which intersects $\{i, \bar{i}\}$ for each $i$ gives rise to a locus $V_T$ in $(\mathbb{P}^1)^N$. 
We have
$$V(\operatorname{in}_{<}(f_1^h), \dotsc, \operatorname{in}_{<}(f_r^h)) = \bigcup_{T \text{ facet of } \widetilde{\Delta}, \, T \cap \{i, \bar{i}\} \not= \emptyset \text{ for all }i} V_T.$$
As $\{\operatorname{spr}(f) : 0 \not= f \in I_X\} = \{\operatorname{spr}(\operatorname{in}_{<}(f^h)) : 0\not= f \in I_X \}$, Lemma~\ref{lem:spreadprojection} implies that $\pi_S(V(\operatorname{in}_{<}(\mathcal{I}_X^h))) = (\mathbb{P}^1)^S$ for each $S \in \Delta$. This implies that, for each facet $F$ of $\Delta$, there is a face $\widetilde{F}$ of $\widetilde{\Delta}$ such that $F = \{i : \{i, \bar{i}\} \subset \widetilde{F}\}$. Then Lemma~\ref{lem:tildelift} gives that this assignment is a bijection between facets of $\widetilde{\Delta}$ and facets of $\Delta$. In particular, there are no associated primes of $(\operatorname{in}_{<}(f_1^h), \dotsc, \operatorname{in}_{<}(f_r^h))$ which are contained in the irrelevant ideal. 
Because $V(\operatorname{in}_{<}(f_1^h), \dotsc, \operatorname{in}_{<}(f_r^h))$ is reduced, we deduce that $V(\operatorname{in}_{<}(\mathcal{I}_X^h)) = V(\operatorname{in}_{<}(f_1^h), \dotsc, \operatorname{in}_{<}(f_r^h))$. Indeed,  for every proper closed subscheme $Y$ of $V(\operatorname{in}_{<}(f_1^h), \dotsc, \operatorname{in}_{<}(f_r^h))$, there is some facet $S$ of $\Delta$ such that $\pi_S(Y) \not= (\mathbb{P}^1)^S$: for each facet $S$ of $\Delta$, $V_{\widetilde{S}}$ is the unique component of $V(\operatorname{in}_{<}(f_1^h), \dotsc, \operatorname{in}_{<}(f_r^h))$ which dominates $(\mathbb{P}^1)^S$. Therefore $\{f_1^h, \dotsc, f_r^h\}$ is a Gr\"{o}bner basis for $\mathcal{I}_X^h$. 
Dehomogenizing gives that $\{f_1, \dotsc, f_r\}$ is a Gr\"{o}bner basis with respect to $<$ (see, e.g., \cite[Exercise 15.39]{Eis95}), and so it is a universal Gr\"{o}bner basis. 

Suppose that $X$ is pure dimensional, so $\mathfrak{F}_{\overline{X}}(z)$ is defined. The multidegree of $\overline{X}$ is the same as the multidegree of $\operatorname{in}_{<}(\mathcal{I}_X^h)$ \cite[Theorem 1.7.1]{KM05}. For $T \subset [N, \bar{N}]$ with $T \cap \{i, \bar{i}\} \not= \emptyset$ for each $i$, we have 
$$\mathfrak{F}_{V_T}(z) = \prod_{\{i, \bar{i}\} \not \subset T} z_i, \quad \text{and} \quad \mathfrak{F}_{\overline{X}}(z) = \sum_{T \text{ facet of }\widetilde{\Delta}} \mathfrak{F}_{V_T}(z).$$ In particular, because the facets of $\widetilde{\Delta}$ are in bijection with the facets of $\Delta$, the multidegrees of the $V_T$ appearing are distinct. Therefore the coefficients in $\mathfrak{F}_{\overline{X}}(z)$ are all either $0$ or $1$. 
\end{proof}

\begin{remark}
Under the hypotheses of Theorem~\ref{thm:grobnercriterion}, $\operatorname{in}_\omega(I_X)$ is generated by $\{\operatorname{in}_\omega(f_1), \dotsc, \operatorname{in}_\omega(f_r)\}$ for any weight $\omega$. Indeed, if we choose some $\omega'$ corresponding to a monomial order, then $\operatorname{in}_{\omega'}(\operatorname{in}_\omega(I_X)) = \operatorname{in}_{\omega + \epsilon \omega'}(I_X)$ for some $\epsilon$ sufficiently small, and $\operatorname{in}_{\omega + \epsilon \omega'}(I_X)$ is a monomial ideal \cite[Proposition 1.11, Proposition 1.13]{SturmfelsGrobner}. Because $\omega'$ corresponds to a monomial order, the fact that $\{\operatorname{in}_{\omega'}(\operatorname{in}_\omega(f_1)), \dotsc, \operatorname{in}_{\omega'}(\operatorname{in}_\omega(f_r))\}$ generates $\operatorname{in}_{\omega + \epsilon \omega'}(I_X)$ implies the claim. 

A similar argument works in the setting of valued Gr\"{o}bner bases in the sense of \cite{ChanMaclagan}: if $\operatorname{in}_\omega(I_X)$ is a monomial ideal, then the proof of Theorem~\ref{thm:grobnercriterion} shows that the $f_i$ are a Gr\"{o}bner basis. At least if $I_X$ is homogeneous, then $\{\operatorname{in}_\omega(f_1), \dotsc, \operatorname{in}_{\omega}(f_r)\}$ generates $\operatorname{in}_{\omega}(I_X)$ even if $\operatorname{in}_\omega(I_X)$ is not monomial, using \cite[Corollary 2.4.10]{MaclaganSturmfels} and \cite[Remark 2.11]{ChanMaclagan}.
\end{remark}

\begin{remark}\label{rem:lex}
Theorem~\ref{thm:grobnercriterion} and Proposition~\ref{prop:pure} imply that if squarefree-supported polynomials $f_1, \dotsc, f_r$ are a Gr\"{o}bner basis for an ideal with respect to every lexicographic term order, then they are a universal Gr\"{o}bner basis for that ideal. This does not hold for arbitrary $f_1, \dotsc, f_r$. For example, $a_1^2a_2^2 - a_1^3 - a_2^3, a_1a_2, a_1^4, a_2^4$ are a Gr\"{o}bner basis for the ideal they generate in $\mathbb{C}[a_1, a_2]$ with respect to either lexicographic order, but they are not a Gr\"{o}bner basis with respect to a graded lexicographic order. We thank Elizabeth Pratt for showing us this example. 
\end{remark}

Using ideas from the proof of Theorem~\ref{thm:grobnercriterion}, we can relate varieties $X \subseteq \mathbb{A}^N$ whose ideals have a squarefree-supported universal Gr\"{o}bner basis to {Cartwright--Sturmfels} ideals, a class of ideals introduced in \cite{CDGMaximal,CDGCS,CDG}. A multihomogeneous monomial ideal $\mathcal{I}$ in $\mathbb{C}[a_1, \dotsc, a_N, b_1, \dotsc, b_N]$ is \emph{Borel-fixed} if for every monomial $m \in \mathcal{I}$ and $i$ such that $m \cdot a_i/b_i$ is a polynomial, we have $m \cdot a_i/b_i \in \mathcal{I}$. A multihomogeneous ideal is \emph{Cartwright--Sturmfels} if it has the same $\mathbb{Z}^N$-graded Hilbert function as a Borel-fixed monomial ideal.

\begin{proposition}\label{prop:CSiff}
Let $X$ be a closed subscheme of $\mathbb{A}^N$. Then $I_X$ has a squarefree-supported universal Gr\"{o}bner basis if and only if $\mathcal{I}_X^h$ is a Cartwright--Sturmfels ideal. 
\end{proposition}

\begin{proof}
First suppose that $\mathcal{I}_X^h$ is Cartwright--Sturmfels. By \cite[Proposition 2.4(6)]{CDGsurvey}, $\mathcal{I}_X^h$ has a universal Gr\"{o}bner basis consisting of squarefree-supported polynomials. We can take them to be multihomogenizations of polynomials in $I_X$. Dehomogenizing these polynomials gives a universal Gr\"{o}bner basis for $I_X$. 

Now suppose that $I_X$ has a squarefree-supported universal Gr\"{o}bner basis $f_1, \dotsc, f_r$ consisting of nonzero polynomials. Then $f_1, \dotsc, f_r$ satisfy the hypothesis of Theorem~\ref{thm:grobnercriterion} by Proposition~\ref{prop:pure}. Choose a term order on $\mathbb{C}[a_1, \dotsc, a_N, b_1, \dotsc, b_N]$. By the proof of Theorem~\ref{thm:grobnercriterion}, the initial ideal is the Stanley--Reisner ideal of a full choice of lifts of $\Delta := \Delta(f_1, \dotsc, f_r)$. If $\widetilde{\Delta}$ arises from a full choice of lifts, then the elements $a_1 - b_1, \dotsc, a_N - b_N$ form a regular sequence on the Stanley--Reisner ring $\mathbb{C}[\widetilde{\Delta}]$, and the quotient is the Stanley--Reisner ring of $\Delta$. This implies that the multigraded Hilbert function of the Stanley--Reisner ring of $\widetilde{\Delta}$ is the same as the multigraded Hilbert function of any complex arising from a full choice of lifts.

Let $\mathcal{J}$ be the ideal obtained by extending the Stanley--Reisner ideal of $\Delta$ to $\mathbb{C}[a_1, \dotsc, a_N, b_1, \dots, b_N]$. This is the Stanley--Reisner ideal of a full choice of lifts, so the multigraded Hilbert function of $\mathcal{I}_X^h$ is the same as the multigraded Hilbert function of $\mathcal{J}$. As $\mathcal{J}$ is radical and Borel-fixed, this means that $\mathcal{I}_X^h$ is Cartwright--Sturmfels. 
\end{proof}

The proof of Proposition~\ref{prop:CSiff} also shows that if $I_X$ has a squarefree-supported universal Gr\"{o}bner basis, then $\mathcal{I}_X^h$ 
is also a \emph{Cartwright--Sturmfels*} ideal in the sense of \cite{CDGCS,CDG}, this gives it remarkable homological properties. For example, the graded Betti numbers of $\mathcal{I}_X^h$ are the same as those of any initial ideal of $\mathcal{I}_X^h$ \cite[Proposition 1.9(3)]{CDG}.

\medskip

We give a description of the fine multidegree of $\overline{X}$ in terms of Gr\"{o}bner degenerations.  This description is not needed in the sequel. 

\begin{proposition}\label{prop:finegrobner}
Let $X$ be a closed subscheme of $\mathbb{A}^N$ whose irreducible components all have codimension $d$. 
For each set $S \subset [N]$ of size $d$, the coefficient of $z^S$ in $\mathfrak{F}_{\overline{X}}(z)$ is the largest multiplicity with which the ideal $(a_i : i \in S)$ appears as an associated prime in a Gr\"{o}bner degeneration of $I_X$. 
\end{proposition}

We first outline the geometry of Proposition~\ref{prop:finegrobner}. For a term order $<$, the coefficient of $z^S$ in $\mathfrak{F}_{\overline{X}}(z)$ records the number of components (counted with multiplicity) of $V(\operatorname{in}_{<} (\mathcal{I}_X^h))$ which have multidegree $S$. We have
$$V(\operatorname{in}_<(I_X)) = V(\operatorname{in}_{<}(\mathcal{I}_X^h)) \cap \mathbb{A}^N,$$
so this coefficient gives an upper bound on the multiplicity of $V(a_i : i \in S)$ in $V(\operatorname{in}_<(I_X))$. If we choose a lexicographic order that starts with $S$, then all components of $V(\operatorname{in}_{<}(\mathcal{I}_X^h))$ with multidegree $S$ intersect $\mathbb{A}^N$, so this bound is attained. 

\begin{proof}[Proof of Proposition~\ref{prop:finegrobner}]
Let $U = S^c$. 
Let $<$ be a term order on $\mathbb{C}[a_1, \dotsc, a_N]$. As in the proof of Theorem~\ref{thm:grobnercriterion}, we extend $<$ to $\mathbb{C}[a_1, \dotsc, a_N, b_1, \dotsc, b_N]$. Then the multidegree of $\operatorname{in}_{<}(\mathcal{I}_X^h)$ is equal to the multidegree of $\mathcal{I}_X^h$ \cite[Theorem 1.7.1]{KM05}. The multidegree of $\operatorname{in}_{<}(\mathcal{I}_X^h)$ is equal to the sum of the multiplicities of the components of $V(\operatorname{in}_{<}(\mathcal{I}_X^h))$ which are of the form $V_T$ for some $T$ with $U = \{i : \{i, \bar{i}\} \subset T\}$. The multiplicity with which $(a_i : i \in S)$ appears as an associated prime in $\operatorname{in}_{<}(I_X)$ is equal to the multiplicity with which $V_{U \cup \overline{U} \cup \overline{S}}$ appears. This shows that $\llbracket z^S\rrbracket \mathfrak{F}_{\overline{X}}(z)$ is an upper bound for the multiplicity of $(a_i : i \in S)$.

To show that this upper bound is attained, it suffices to construct a term order for which the multiplicity of $V_{U \cup \overline{U} \cup \widetilde{S}}$ is $0$ if $\widetilde{S}$ is a lift of $S$ which is not $\overline{S}$. We may assume that $S = \{1, \dotsc, d\}$, and let $<$ be the lexicographic term order induced by $1 > 2 > \dotsb > N$. 

Let $\{f_1^h, \dotsc, f_r^h\}$ be a Gr\"{o}bner basis for $<$ consisting of multihomogenizations of elements of $I_X$. For existence of such a Gr\"{o}bner basis, see, e.g., \cite[Exercise 15.39]{Eis95}. Let $\widetilde{S}$ be a lift of $S$ which is not $\overline{S}$, and let $j$ be the smallest element of $S \cap \widetilde{S}$. The choice of term order guarantees that if $\operatorname{in}_{<}(f_k^h)$ contains $b_j$, then it must contain $a_i$ for some $i < j$. In particular, as $i \not \in \widetilde{S}$, if $\operatorname{in}_{<}(f_k^h)$ vanishes on $V_{U \cup \overline{U} \cup \widetilde{S}}$, then it also vanishes on $V_{U \cup \overline{U} \cup \widetilde{S} \cup \{\bar{j}\}}$. This means $V_{U \cup \overline{U} \cup \widetilde{S}}$ is not an irreducible component of $V(\operatorname{in}_{<}(f_1^h), \dotsc, \operatorname{in}_{<}(f_r^h))$.
\end{proof}

In particular, Proposition~\ref{prop:finegrobner} shows that all coefficients of $\mathfrak{F}_{\overline{X}}(z)$ are either $0$ or $1$ if and only if all Gr\"{o}bner degenerations of $X$ are generically reduced. 

The proof of Proposition~\ref{prop:finegrobner} shows that, for any $S$, $\llbracket z^S\rrbracket\mathfrak{F}_{w}(z)$
is the multiplicity of $\mathbb{A}^{S^c}$ in $V(\operatorname{in}_<(I_w))$, where $<$ is a lexicographic term order in which $S$ occurs first. In some cases, this can be used to compute $\llbracket z^S\rrbracket \mathfrak{F}_{w}(z)$. For example, if there is an antidiagonal term order which is a lexicographic term order in which $S$ occurs first, then \cite{KM05} shows that $\llbracket z^S\rrbracket \mathfrak{F}_{w}(z)$ is $1$ if $S$ is the set of crosses in a reduced pipe dream for $w$, and is $0$ otherwise.

\begin{remark}\label{rem:CS}
If $X$ is integral and all coefficients of $\mathfrak{F}_{\overline{X}}(z)$ are $0$ or $1$, then it follows from \cite{BrionMultiplicity} that $\mathcal{I}_X^h$ is a \emph{Cartwright--Sturmfels} ideal in the sense of \cite{CDGCS,CDG}. By Proposition~\ref{prop:CSiff}, $I_X$ has a squarefree-supported universal Gr\"{o}bner basis in this case, so there will be $f_1, \dotsc, f_r \in I_X$ so that the hypothesis of Theorem~\ref{thm:irreducible} holds. In particular, in the case when $X$ is integral, the following are equivalent:
\begin{enumerate}
    \item $X$ has a squarefree-supported universal Gr\"{o}bner basis. 
    \item All Gr\"{o}bner degenerations of $X$ are reduced.
    \item All Gr\"{o}bner degenerations of $X$ are generically reduced. 
    \item The circuit polynomials of $X$ are squarefree-supported. 
    \item $\mathfrak{F}_{\overline{X}}(z)$ has all coefficients $0$ or $1$.
\end{enumerate}
\end{remark}

\begin{example}
Let $X$ be a closed subscheme of $\mathbb{A}^N$ with the property that $I_X$ has a squarefree-supported universal Gr\"{o}bner basis. Then for any $i \in [N]$ and $\lambda \in \mathbb{C}$, the ideal of $X \cap V(a_i - \lambda)$ has a universal Gr\"{o}bner basis obtained by setting $a_i$ equal to $\lambda$ in the squarefree-supported universal Gr\"{o}bner basis of $I_X$. Indeed, by the condition in Theorem~\ref{thm:grobnercriterion}, we may translate $X$ without changing whether $I_X$ has a squarefree-supported universal Gr\"{o}bner basis, so we may take $\lambda = 0$. Then this follows from the proof of \cite[Proposition 5.4]{BoocherDeterminant}.
\end{example}

\subsection{Gr\"{o}bner bases for the varieties of rank 1 matrices}\label{ssec:determinantal}

We now use Theorem~\ref{thm:irreducible} to give a short proof of the known universal Gr\"{o}bner basis for the ideal generated by the $2 \times 2$ minors of a generic matrix. 

Let $X \subset \mathbb{A}^{pq}$ be the variety of $p \times q$ matrices of rank at most $1$. It is well-known (and follows easily from the proof of Proposition~\ref{prop:rank1GB}) that $X$ is irreducible of dimension $p + q - 1$.
The ideal $I_X$ of $X$ is generated by $2 \times 2$ minors, but these are usually \emph{not} a universal Gr\"{o}bner basis. 
We will need to use the $\mathrm{merge}$ operation from the introduction in order to construct new relations. 

For this, it is helpful to identify $[p] \times [q]$ with the edges of the complete bipartite graph $K_{p,q}$. Then the spread of a $2 \times 2$ minor corresponds to a $4$-cycle in this graph. Suppose we have squarefree-supported elements $f_1, f_2 \in I_X$ corresponding to subgraphs $G_1$ and $G_2$.  Then they satisfy the hypothesis used to define $\mathcal{R}_w$ if and only if there is a unique edge $(\alpha, \beta)$ in both $G_1$ and $G_2$, and $\alpha$ and $\beta$ are the only vertices used by both $G_1$ and $G_2$. In this case, the spread of $\operatorname{merge}_{(\alpha, \beta)}(f_1, f_2)$ corresponds to the subgraph $G_1 \cup G_2 \setminus \{(\alpha, \beta)\}$.

In particular, every cycle in $K_{p,q}$ can be obtained in this way: for $k \ge 2$, each $(2k+2)$-cycle is obtained by merging a $2k$-cycle with a $4$-cycle. Explicitly, given a cycle with edges $(i_1, j_1), (i_2, j_1), \dotsc, (i_k, j_k), (i_1,j_k)$, we have the relation
\begin{equation}\label{eq:cycle}
\prod_{\ell=1}^{k} a_{i_\ell,j_{\ell}} - \prod_{\ell=1}^{k} a_{i_{\ell+1},j_{\ell}} \in I_X, \quad \text{ where }i_{k+1} = i_1.
\end{equation}

In particular, Proposition~\ref{prop:spreadmultidegree} shows that if $S \subset [p] \times [q]$ has $\llbracket z^S\rrbracket \mathfrak{F}_{\overline{X}}(z) > 0$, then $S^c$ corresponds to a forest in $K_{p,q}$. As $\dim X = p + q - 1$, we must have $|S^c| = p + q - 1$, so $S^c$ must be a spanning tree. 

\begin{proposition} \cite[Proposition 4.11 and 8.11]{SturmfelsGrobner} \cite[Proposition 10.1.11]{Villarreal} \label{prop:rank1GB}
Let $X \subset \mathbb{A}^{pq}$ be the variety of $p \times q$ matrices of rank at most $1$. Then the relations in \eqref{eq:cycle} are a universal Gr\"{o}bner basis for $I_X$. 
\end{proposition}

\begin{proof}
By Theorem~\ref{thm:irreducible}, it suffices to show that if $S^c$ is a spanning tree of $K_{p,q}$, then $\llbracket z^S\rrbracket \mathfrak{F}_{\overline{X}}(z) > 0$. Let $U = S^c$.
Every $p \times q$ matrix of rank at most $1$ can be written as a product of a $p \times 1$ matrix and a $1 \times q$ matrix. This gives a surjective map
$$f \colon \mathbb{A}^{p} \times \mathbb{A}^q \to X \subset \mathbb{A}^{pq}, \quad a_{i,j} \mapsto u_i \otimes v_j,$$
where $(u_1, \dotsc, u_p)$ are the coordinates on $\mathbb{A}^p$ and $(v_1, \dotsc, v_q)$ are the coordinates on $\mathbb{A}^q$. 

By Lemma~\ref{lem:proj}, $\llbracket z^S\rrbracket \mathfrak{F}_{\overline{X}}(z) > 0$ if and only if the projection of $X$ onto $\mathbb{A}^{U}$ is dominant. This happens if and only if the composition  $\mathbb{A}^{p} \times \mathbb{A}^{q} \to \mathbb{A}^U$ is dominant. To prove that this is dominant when $U$ is a spanning tree, it suffices to show that the Jacobian matrix has full rank. The Jacobian of the composition $\mathbb{A}^{p} \times \mathbb{A}^q \to \mathbb{A}^U$ is a $(p + q) \times (p + q - 1)$ matrix whose rows are labeled by vertices of $K_{p,q}$ and whose columns are labeled by edges in $U$. 
In the column labeled by an edge $(i, j)$, as $\operatorname{d}(u_iv_j) = u_i \operatorname{d}v_j + v_j \operatorname{d}u_i$, we have an entry of $v_j$ in the row labeled by the $i$th vertex, an entry of $u_i$ in the row labeled by the $j$th vertex, and $0$ for all other entries. This is the incidence matrix of the subgraph corresponding to $U$, and so it has full rank if (and only if) $U$ is a spanning tree. 
\end{proof}

\begin{remark}
The key to the simple proof of Proposition~\ref{prop:rank1GB} is the explicit description of the subsets of $[p] \times [q]$ which avoid the spreads of the relations: they correspond to forests in a bipartite graph. We do not know a simple description for the sets which avoid the spreads of elements of the sets $\mathcal{R}_w$ in Theorem~\ref{thm:main}.
\end{remark}

\begin{remark}
The same argument can be used to give a short proof that the maximal minors are a universal Gr\"{o}bner basis for the ideal that they generate \cite{SturmfelsZelevinsky,BernsteinZelevinsky}. Indeed, the Jacobian matrix of the map obtained from the fact that a matrix of rank at most $p - 1$ can be written as a product of a $p \times (p-1)$ matrix and a $(p-1) \times q$ matrix is very simple, see \cite[Theorem 1.1(3), Corollary 3.2]{BDGGL}.
\end{remark}

\begin{remark}
Let $X \subset \mathbb{A}^{16}$ be the variety of $4 \times 4$ matrices of rank at most $2$. Then one can check that $[z_{1,1}z_{2,2}z_{3,3}z_{4,4}] \mathfrak{F}_{\overline{X}}(z_{i,j}) = 2$ (in any characteristic), so Theorem~\ref{thm:irreducible} cannot be applied to $X$. Using Proposition~\ref{prop:patterncontainment}, one can check that the fine multidegree polynomial of the locus of $p \times q$ matrices of rank at most $r$ has all coefficients equal to $0$ or $1$ if and only if $r \in \{0, 1, p-1, q-1, p, q\}$.
\end{remark}

\section{Geometric tools}\label{sec:geom}

In this section, we develop a number of geometric tools for studying fine Schubert polynomials. As mentioned in the introduction, we do not know of an inductive formula for fine Schubert polynomials, so none of these tools can be used to compute fine Schubert polynomials in general. However, in the proof of Theorem~\ref{thm:main} we show that, taken together, these tools are enough to compute the fine Schubert polynomial of certain permutations. 

We set up some notation that we will use for the rest of the paper. For $w\in S_n$ and $(i,j)\in [n]^2$, define $w\cdot (i,j):= (w(i),j)$ and $(i,j)\cdot w:= (i,w(j))$. For $U \subset [n]^2$, set $w \cdot U := \{w \cdot (i, j) : (i, j) \in U\}$, and define $U \cdot w$ similarly. 
Set $U_{A, B} = U \cap (A \times B)$. When $A$ or $B$ is $[n]$ we use $*$ as a shorthand, so $U_{i,*}:=U\cap (\{i\}\times [n])$.
For $v \in S_n$, let $v \cdot X_w$ be the variety in $\mathbb{A}^{n^2}$ obtained by applying the automorphism of $\mathbb{A}^{n^2}$ which applies the permutation $v$ to the rows, and let $X_w \cdot v$ be the variety obtained by permuting the columns via $v$. For $1\le i<j\le n$, let $t_{i,j}$ be the transposition $(i, j)$, and set $s_i = t_{i, i+1}$.

\medskip

We recall a consequence of Fulton's approach to intersection theory. Let $X$ be a smooth projective variety, and let $Y, Z$ be subvarieties whose irreducible components all have codimension $k$, $\ell$, respectively. Let $[Y], [Z] \in H^*(X)$ be the fundamental classes of $Y$ and $Z$ in the cohomology of $X$. Let $C_1, \dotsc, C_r$ be the irreducible components of $Y \cap Z$. By \cite[Section 6.1]{Fulton}, we have an equality
$$[Y] \smallsmile [Z] = \sum_{i=1}^{r} m_i \alpha_i,$$
where each $m_i$ is positive and $\alpha_i \in H^{2(k + \ell)}(X)$ is pushed forward from $C_i$. In particular, if $C_i$ has the expected dimension, i.e., $C_i$ has codimension $k + \ell$, then $\alpha_i = [C_i]$. In this case, $m_i$ is bounded above by the multiplicity of $C_i$ \cite[Proposition 7.1]{Fulton}, and so if $C_i$ is generically reduced, then $m_i = 1$. 
If $X$ has globally generated tangent bundle, e.g., if $X$ is a product of projective spaces, then each $\alpha_i$ is effective \cite[Theorem 12.2(a)]{Fulton}.

We will frequently use the above properties in the following way: for some subvariety $Z$ of $(\mathbb{P}^1)^{n^2}$, we will find an irreducible component $C$ of $Z \cap \overline{X}_w$ of the expected dimension and use the above to obtain an equality $[Z] \smallsmile [\overline{X}_w] = [C] + E$, where $E$ is the class of an effective cycle, so $E$ is a non-negative sum of classes of coordinate subspaces. This gives a lower bound on the coefficients of $\mathfrak{F}_{w}(z)$. We do not need to check that all other irreducible components of $Z \cap \overline{X}_w$ have the expected dimension. 

\medskip

We also recall one property of matrix Schubert varieties that will be used several times. For $w \in S_n$, let $p_w \in \mathbb{A}^{n^2}$ be the permutation matrix corresponding to $w$. Then $p_w \in X_w$.

\begin{proposition}\cite[Lemma 4.2]{KnutsonCotransition}\label{prop:tangent}
The tangent space to $p_w$ in $X_w$ is spanned by the matrix entries \emph{not} in the Rothe diagram $D(w)$. 
\end{proposition}

In particular, the dimension of the tangent space of $X_w$ at $p_w$ is equal to $n^2 - \ell(w)$, which is the dimension of $X_w$. This means that if $Y \subset \mathbb{A}^{n^2}$ is a closed subscheme which contains $X_w$ and has the property that the tangent space of $Y$ at $p_w$ has dimension $n^2 - \ell(w)$, then $X_w$ is an irreducible component of $Y$.

Finally, we state a general tool which will be used in this section and throughout Section~\ref{sec:proof}, which states that swapping certain rows or columns induces an automorphism of the matrix Schubert variety. 

\begin{proposition}\label{prop:swap}
Suppose that $w(i) < w(i+1)$. Then $s_i \cdot X_w = X_w$. 
\end{proposition}
\begin{proof}
The assumption that $w(i) < w(i+1)$ implies that the essential set $\operatorname{ess}(w)$ does not intersect the $i$th row. This means that multiplying on the left by $s_i$ preserves the Fulton generators, and so it preserves $X_w$. 
\end{proof}

Similarly, if $w^{-1}(i) < w^{-1}(i+1)$, then $X_w \cdot s_i = X_w$.

\subsection{Pattern avoidance}

We now prove Proposition~\ref{prop:patterncontainment}. Choose some $k \in [n]$, and let $v:=\del_k(w) \in S_{n-1}$ be the permutation given by the relative ordering of $w(1), \dotsc, w(k-1), w(k+1), \dotsc, w(n)$. 
View $\mathbb{A}^{(n-1)^2}$ as $\operatorname{Spec} \mathbb{C}[a_{i,j} : i \in [n] \setminus \{k\}, \, j \in [n] \setminus \{w(k)\}]$. 
We identify $\mathbb{A}^{(n-1)^2}$ with a subset of $\mathbb{A}^{n^2}$ via the map
$$\mathbb{C}[a_{i,j}] \to \mathbb{C}[a_{i,j} : i \in [n] \setminus \{k\}, \, j \in [n] \setminus \{w(k)\}] \quad   a_{k, w(k)} \mapsto 1, \, a_{k,j}, \, a_{j,w(k)} \mapsto 0, \text{ and } a_{i,j} \mapsto a_{i,j} \text{ otherwise.}$$
In particular, we identify $X_{v}$ with a subset of $\mathbb{A}^{n^2}$. Let $\iota_k:[n-1]\times [n-1]\to [n]\times [n]$ denote the corresponding injection on coordinates.
Namely,
\[\iota_k(i,j)=\begin{cases}(i,j) &\text{ if }i<k\text{ and } j<w(k)\\
(i,j+1) &\text{ if }i<k\text{ and } j\ge w(k)\\
(i+1,j) &\text{ if }i\ge k\text{ and } j<w(k)\\
(i+1,j+1) &\text{ if }i\ge k\text{ and } j\ge w(k)\\
\end{cases}.\]

\begin{lemma}\label{lemma:containment}
We have 
$$X_{w} \cap V(a_{k, w(k)} - 1) \cap V(a_{k,j} : j\not= w(k))  \cap V(a_{j, w(k)} : j \not= k) = X_{v}.$$
\end{lemma}

\begin{proof}
If we make the above substitutions, then each Fulton generator of $I_w$ turns into a Fulton generator of $I_{v}$, and all Fulton generators of $I_v$ are obtained in this way. 
\end{proof}

Set $Z=V(a_{k,j} : (k, j) \not \in D(w), j\not= w(k)) \cap V(a_{j, w(k)} : (j, w(k)) \not \in D(w), j \not= k)\cap V(a_{k,w(k)}-1)$.

\begin{proof}[Proof of Proposition~\ref{prop:patterncontainment}]
By Lemma~\ref{lemma:containment}, $\overline{X}_v$ is contained in $\overline{X}_w \cap \overline{Z}$. It follows from Proposition~\ref{prop:tangent} that the tangent space at $p_w$ of $\overline{X}_w \cap \overline{Z}$ is equal to the tangent space of $\overline{X}_v$ at $p_w$. This implies that $\overline{X}_v$ is an irreducible component of $\overline{X}_w \cap \overline{Z}$. 

Recall that $M = \prod z_{i,j}$, where the product is over pairs $(i, j)$ where $(i, j) \in D(w)$ and $i = k$ or $j = w(k)$. 
Let $M' = \prod_{i = k \text{ or } j = w(k)} z_{i,j}$.  Note that $[\overline{Z}] = M'/M \in H^*((\mathbb{P}^1)^{n^2})$. As $\overline{X}_v$ is a component of $\overline{X}_w \cap \overline{Z}$, the discussion at the start of Section~\ref{sec:geom} gives that
$$[\overline{X}_w] \smallsmile [\overline{Z}] = [\overline{X}_v] + E,$$
where $E$ is an effective class. Because we are viewing $\overline{X}_v$ as a subspace of $(\mathbb{P}^1)^{n^2}$ via an embedding into a coordinate subspace, we have $[\overline{X}_v] = M' \cdot \mathfrak{F}_{v}(z_{i,j})_{i \in [n] \setminus \{k\}, \, j \in [n] \setminus \{w(k)\}}$. 
This gives that
$$M'/M \cdot \mathfrak{F}_{w}(z) = M' \cdot \mathfrak{F}_{v}(z_{i,j})_{i \in [n] \setminus \{k\}, j \in [n] \setminus \{w(k)\}} + E',$$
where $E'$ has non-negative coefficients. This implies the result. 
\end{proof}

\begin{remark}\label{rem:schubertpattern}
The same argument can be used to give a new proof of \cite[Theorem 1.2]{FMS}, which gives a similar monotonicity result for Schubert polynomials, by taking the closure of $X_w$ in $(\mathbb{P}^n)^n$ instead of $(\mathbb{P}^1)^{n^2}$; see Lemma~\ref{lem:classYw}. 
\end{remark}

\subsection{Cotransition}
In this section, we prove a version of Knutson's cotransition formula for Schubert polynomials \cite{KnutsonCotransition}. Let $w_0 = (n, n-1, \dotsc, 1)$ be the long element of $S_n$. We write $v \gtrdot w$ to mean that $v$ covers $w$ in Bruhat order. 

Given $w \in S_n$, we say that we can \emph{cotransition at} $k \in [n]$ if there is no $a < k$ such that $w  t_{a,k} \gtrdot w$. If this is the case, then let $L_k(w): = \{w  t_{k,b} : k < b, \, w  t_{k,b} \gtrdot w\}$. For example, if $w \not= w_0$ and $k$ is the minimal element such that $k + w(k) \le n$, then we can cotransition at $k$, and $L_k(w) = \{w' \in S_n : w' \gtrdot w, \, w'(k) \not= w(k)\}$. 

In \cite[Section 4]{KnutsonCotransition}, Knutson proves a special case of the following result, and his argument generalizes without modification. 

\begin{proposition}\label{prop:geometriccotransition}
Suppose we can cotransition at $k$. Then
$$X_w \cap V(a_{k, w(k)}) = \bigcup_{w' \in L_k(w)} X_{w'}.$$
\end{proposition}

We use this to deduce that fine Schubert polynomials behave well with respect to cotransition. 

\begin{proposition}\label{prop:cotransition}
Suppose we can cotransition at $k$. Then for any $S$ not containing $(k, w(k))$, we have
$$\llbracket z^S\rrbracket \mathfrak{F}_{w}(z) \ge \sum_{w' \in L_k(w)} \llbracket z^{S \cup \{(k, w(k))\}}\rrbracket \mathfrak{F}_{w'}(z).$$
\end{proposition}

\begin{proof}
In the decomposition
$$X_w \cap V(a_{k, w(k)}) = \bigcup_{w' \in L_k(w)} X_{w'},$$
each $X_{w'}$ is codimension $1$ in $X_w$ and is an irreducible component of the intersection. Therefore, for each $w' \in L_k(w)$, $\overline{X}_{w'}$ is an irreducible component of the expected dimension in $\overline{X}_w \cap V(a_{k, w(k)})$. The discussion at the start of Section~\ref{sec:geom} implies the desired bound. 
\end{proof}

\subsection{Deleting a special column}

In this section, we will develop tools that allow us to relate $\mathfrak{F}_{w}(z)$ to the fine Schubert polynomial of a permutation $w'$ that is obtained by ``deleting a column'' from the Rothe diagram of $w$. 
Define the \emph{dominant part} of a 
permutation $w\in S_n$ to be 
\[\dom(w):=\{(i,j)\in D(w): (k,w(k))\not\in [i]\times [j] \text{ for all }k\in[n]\}.\]
Note that $\dom(w)=\emptyset$ if $w(1)=1$.
In other words, $\dom(w)$ is the  region of $D(w)$ connected to $(1,1)$ when nonempty. 

We say $c$ is a \textbf{special column} if the following
conditions are satisfied:
\begin{enumerate}
    \item[(a)]  $D(w)_{*,c}\setminus\dom(w)\neq \emptyset$;
    \item[(b)]  If $r$ is maximal such that $(r,c)\in D(w)$,
    then $D(w)_{[r+1,n],[c,n]}=\emptyset$. Namely,
    there are no squares weakly southeast of
    $(r+1,c)$ in the Rothe diagram.
\end{enumerate}
We will construct a permutation $\del_{*,c}(w)$ such that 
\begin{equation}
\label{eq:shift-left-one}
D(\del_{*,c}(w))=\{(i,j)\in D(w):j< c\}\cup \{(i,j-1):j>c, (i,j)\in D(w)\}.    
\end{equation}
In other words, the diagram for $\del_{*,c}(w)$ is obtained from $D(w)$ by deleting the boxes in column $c$ and shifting every box to the right of column $c$ by one column to the left. See Figure~\ref{fig:deletecols} for an example.

When $(r,c+1)\not\in D(w)$, namely $(r,c)$ is a 
southeast corner, define
\begin{equation}
    \del_{*,c}(w)(i):=\begin{cases}
        w(i) & \text{ if }w(i)<c \\
        w(i)-1 &\text{ if }i\le r,\ w(i)>c\\
        \min([n]\setminus ([c]\cup \del_{*,c}(w)([i-1]))) &\text{ if }i>r,\ w(i)\ge c
    \end{cases}
\end{equation}
Otherwise, recursively define $\del_{*,c}(w):=\del_{*,c+1}(w)$.

\begin{lemma}
The function $\del_{*,c}(w)$ is a permutation that satisfies \eqref{eq:shift-left-one}.
\end{lemma}

\begin{proof}
    It suffices to assume that $(r,c+1)\not\in D(w)$.
    Let $w':=\del_{*,c}(w)$. To see that $w'$ is a permutation, we observe from the definition that $w'(i)\le n$ and  
    $w'(i)\not\in w'([i-1])$ for all $i\in [n]$.

    The first case of the construction ensures that $D(w)$ and $D(w')$ agree in columns $<c$. The second case ensures that in rows $\le r$, the cells in $D(w')$ are the cells in $D(w)$ shifted by one column to the left.
    The third case ensures that there are no boxes 
    weakly southeast of $(r+1,c)$.
\end{proof}

We now analyze the geometry of deleting a special column. For $S \subset [n]^2$, we view $\mathbb{A}^S$ as the coordinate subspace of $\mathbb{A}^{n^2}$ consisting of matrices whose entries are $0$ except possibly for the entries in $S$.
For a permutation $w$, let 
$$\Phi(w) = \{(i, j) : \text{there is }(k, \ell) \in D(w) \text{ with }k \ge i, \, \ell \ge j \},$$
i.e., the subset of $[n] \times [n]$ which is northwest of some square in the Rothe diagram. We call this the \textbf{interesting part}, for the following reason. The spread of each Fulton generator of $I_w$ is contained in $\Phi(w)$, and so $X_w \simeq (X_w \cap \mathbb{A}^{\Phi(w)}) \times \mathbb{A}^{[n]^2 \setminus \Phi(w)}.$ We call $X_w \cap \mathbb{A}^{\Phi(w)}$ the interesting part of $X_w$. The support of every monomial with nonzero coefficient in $\mathfrak{F}_{w}(z)$ is contained in $\Phi(w)$. Note that the dimension of $X_w \cap \mathbb{A}^{\Phi(w)}$ is $|\Phi(w)| - \ell(w)$. 

Let $c$ be a special column of $w$, and let $w' = \del_{*,c}(w)$. Then there is an injection $\iota_{*,c} \colon \Phi(w') \to \Phi(w)$ given by 
$$\iota_{*,c}((i, j)) = \begin{cases} (i, j) & \text{ if }j < c \\ (i, j+1) & \text{ if }j \ge c.\end{cases}$$ 
Let $E = \Phi(w) \setminus (\iota_{*,c}(\Phi(w')) \cup \Phi(w)_{*,c})$; note that $E$ is empty if there is a box of $D(w)$ strictly to the east of the $c$th column and a box of $D(w)$ weakly south of $(r, c)$. 

\begin{lemma}\label{lem:containmentcol}
The map $\mathbb{A}^{\Phi(w') \cup E} \to \mathbb{A}^{\Phi(w)}$ induced by $\iota_{*,c}$ restricts to an isomorphism from $X_{w'} \cap \mathbb{A}^{\Phi(w') \cup E}$ to $X_w \cap \mathbb{A}^{\Phi(w)} \cap V(a_{i, c} : (i, c) \in \Phi(w))$.
\end{lemma}

\begin{proof}
The rank of every square $(i, j)$ in $D(w')$ is the same as the rank of $\iota_{*,c}((i,j))$ in $D(w)$. Setting $a_{i, c} = 0$ for $(i, c) \in \Phi(w)$ kills the Fulton generators whose spreads intersect the $c$th column. Under the identification induced by $\iota_{*,c}$, the remaining Fulton generators for $I_w$ are the same as the Fulton generators for $I_{w'}$. 
\end{proof}

We will now show that, in some situations, we can use Lemma~\ref{lem:containmentcol} to bound the fine Schubert polynomial of $w$ in terms of the fine Schubert polynomial of $w'$. See Example~\ref{ex:cannotdelete} for an example of the subtleties in this type of statement. We use the notation above.

\begin{proposition}\label{prop:deletearbitrary}
Suppose $S \subset \Phi(w)$ satisfies $S_{*,c}=D(w)_{*,c}$ and $|S| = \ell(w)$. Then
$$\llbracket z^S\rrbracket \mathfrak{F}_{w}(z) \ge \llbracket z^{\iota_{*,c}^{-1}(S)}\rrbracket\mathfrak{F}_{w'}(z).$$
\end{proposition}

\begin{proof}
By the discussion at the beginning of Section~\ref{sec:geom}, it suffices to show that
$$X_{w'} \cap \mathbb{A}^{\Phi(w') \cup E} \text{ is an irreducible component of }X_w \cap \mathbb{A}^{\Phi(w)} \cap V(a_{i,c} : (i, c) \in \Phi(w) \setminus D(w)),$$
where we are using the identification in Lemma~\ref{lem:containmentcol}. For this, it suffices to check that there is a point of $X_w \cap \mathbb{A}^{\Phi(w)} \cap V(a_{i,c} : (i, c) \in \Phi(w) \setminus D(w))$ which is contained in $X_{w'} \cap \mathbb{A}^{\Phi(w') \cup E}$ where the dimension of the tangent space of $X_w \cap \mathbb{A}^{\Phi(w)} \cap V(a_{i,c} : (i, c) \in \Phi(w) \setminus D(w))$ is equal to the dimension of $X_{w'} \cap \mathbb{A}^{\Phi(w') \cup E}$. 

Let $p$ be the point of $X_w \cap \mathbb{A}^{\Phi(w)}$ which is $1$ at each coordinate of the form $(k, w(k))$ in $\Phi(w)$ and is $0$ at all other coordinates. Because $X_w$ is the product of the interesting part with some affine factors, Proposition~\ref{prop:tangent} implies that the tangent space of $X_w \cap \mathbb{A}^{\Phi(w)}$ at $p$ is spanned by matrix entries in $\Phi(w) \setminus D(w)$. The ``special column'' condition guarantees that $(w^{-1}(c), c) \not \in \Phi(w)$, so $p \in X_w \cap \mathbb{A}^{\Phi(w)} \cap V(a_{i,c} : (i, c) \in \Phi(w) \setminus D(w))$. From the definition of $\del_{*,c}(w)$, we see that $p \in X_{w'} \cap \mathbb{A}^{\Phi(w') \cup E}$, and the tangent space of $X_w \cap \mathbb{A}^{\Phi(w)} \cap V(a_{i,c} : (i, c) \in \Phi(w) \setminus D(w))$ at $p$ has dimension equal to the dimension of $X_{w'} \cap \mathbb{A}^{\Phi(w') \cup E}$, as desired. 
\end{proof}

We now prove a version of Proposition~\ref{prop:deletearbitrary} where the condition on how $S$ meets $D(w)_{*,c}$ is different. Recall that $w' = \del_{*,c}(w)$.

\begin{proposition}\label{prop:deleteoneoff}
Let $c$ be a special column such that $D(w)_{*,c} \setminus \dom(w)$ is connected.
Suppose $S \subset \Phi(w)$ has $|S_{*,c}| = |D(w)_{*,c}|$, $\dom(w)_{*, c} \subset S$, $|S_{*, c} \setminus D(w)| = 1$, and $|S| = \ell(w)$. Then
$$\llbracket z^S\rrbracket \mathfrak{F}_{w}(z) \ge \llbracket z^{\iota_{*,c}^{-1}(S)}\rrbracket\mathfrak{F}_{w'}(z).$$
\end{proposition}

In particular, if every $(i, c) \in D(w)_{*, c} \setminus \dom(w)$ has $r_{i \times c}(w) = 1$, or if $|D(w)_{*,c} \setminus \dom(w)| = 1$, then either Proposition~\ref{prop:deletearbitrary} or Proposition~\ref{prop:deleteoneoff} gives that $\llbracket z^S\rrbracket \mathfrak{F}_{w}(z) \ge \llbracket z^{\iota_{*,c}^{-1}(S)}\rrbracket \mathfrak{F}_{w'}(z).$

\begin{lemma}\label{lem:coliscomp}
Let $c$ be a special column such that $D(w)_{*,c} \setminus \dom(w)$ is connected.
Suppose $S \subset \Phi(w)$ has $|S_{*,c}| = |D(w)_{*,c}|$, $\dom(w)_{*, c} \subset S$, $S \subset \Phi(w)$, and $|S_{*, c} \setminus D(w)| = 1$. Then, under the identification in Lemma~\ref{lem:containmentcol}, $X_{w'} \cap \mathbb{A}^{\Phi(w') \cup E}$ is an irreducible component of 
$$X_w \cap \mathbb{A}^{\Phi(w)} \cap \bigcap_{\substack{i, \, (i, c) \not \in S \\ (i, c) \in \Phi(w)}}V(a_{i,c}).$$
\end{lemma}

\begin{proof}
Let $(i, c)$ be the unique element of $S_{*,c} \setminus D(w)$, and let $(j, c)$ be the unique element of $D(w)_{*,c} \setminus S$. The assumptions guarantee that $i < j$, $w(i) < w(j)$, and that $(w^{-1}(c), c) \not \in \Phi(w)$. Because $i < j$, the map which adds the $i$th row to the $j$th induces an automorphism of $X_w \cap \mathbb{A}^{\Phi(w)}$ which preserves $X_{w'} \cap \mathbb{A}^{\Phi(w') \cup E}$. 

Let $p$ be the point of $\mathbb{A}^{\Phi(w)}$ which is $1$ at each coordinate of the form $(k, w(k))$ in $\Phi(w)$, is $1$ at $(j, w(i))$, and is $0$ at all other points. Note that $p$ lies in $X_{w'} \cap \mathbb{A}^{\Phi(w') \cup E}$. It follows from Proposition~\ref{prop:tangent} that the tangent space to $X_w \cap \mathbb{A}^{\Phi(w)}$ has dimension $|\Phi(w)| - \ell(w)$, and that each $a_{s,c}$ with $(s, c) \not \in S$ and $(s, c) \in \Phi(w)$ cuts down the dimension of the tangent space at $p$ by $1$. This implies the result. 
\end{proof}

\begin{proof}[Proof of Proposition~\ref{prop:deleteoneoff}]
Lemma~\ref{lem:coliscomp} implies that $\overline{X}_{w'} \cap (\mathbb{P}^1)^{\Phi(w') \cup E}$ is an irreducible component of the expected dimension of 
$$\overline{X}_w \cap (\mathbb{P}^1)^{\Phi(w)} \cap \bigcap_{\substack{i, \, (i, c) \not \in S \\ (i, c) \in \Phi(w)}}V(a_{i,c}).$$
The discussion at the beginning of Section~\ref{sec:geom} implies the result. 
\end{proof}

We say that $r$ is a \textbf{special row} for $w$ if $r$ is a special column for $w^{-1}$. If $r$ is a special row, define $\del_{r, *}(w) = \del_{*, r}(w^{-1})^{-1}$, and define $\iota_{r,*}$ in a similar fashion to $\iota_{*,c}$. As $X_w$ is isomorphic to $X_{w^{-1}}$ via the transpose map, all the results above can be immediately adapted to deleting a special row.

\begin{example}\label{ex:cannotdelete}
Let $w = 145263$, and let $S = \{(1,1), (2,1), (3,3), (4,3), (5,3)\}$. Then the third column is special, and $\del_{*,3}(w) = 134256$. We have $\llbracket z^S\rrbracket \mathfrak{F}_{w}(z) = 0$ because $S^c$ contains the spread of the $2 \times 2$ minor in rows $1$ and $2$ and columns $2$ and $3$, but $\llbracket z^{\iota_{*,3}^{-1}(S)}\rrbracket \mathfrak{F}_{\del_{*,3}(w)}(z) = 1$. This shows that the hypothesis that $D(w)_{*,3} \setminus \dom(w)$ is connected cannot be omitted in Proposition~\ref{prop:deleteoneoff}.
\end{example}

\subsection{Deleting a solid column}
In the proof of Theorem~\ref{thm:main}, we will need to delete another class of columns whose intersection with the Rothe diagram is contained in the dominant part. We will analyze this case by reducing to the case of special columns. 

We  say that a column $c$ is \textbf{solid} if there is some $k > 0$ such that
\begin{enumerate}
    \item  $c+k$ is a special column with $D(w)_{*,c+k}\setminus \dom(w)=\{(r, c+k)\}$ for some row $r$,
    \item for all $0\le k' <k$, $D(w)_{*,c+k'}\subset \dom(w)_{[r-2], c+k'}$.
\end{enumerate}

 We  define a permutation $\del_{*,c}(w) \in S_n$ which ``deletes'' the solid column $c$. 
Define
\[B_i\coloneqq\{j:(i,j)\in D(w), j<c\}\cup \{j-1:(i,j)\in D(w), j>c, (i,j)\neq(r,c+k)\},\]
and let \[\del_{*,c}(w)(i)\coloneqq \min([n]\setminus (B_i\cup \del_{*,c}(w)([i-1]))).\]

Informally, the Rothe diagram of $w'$ is obtained from the Rothe diagram of $w$  by deleting the $c$th column, removing the unique square in $D(w)_{*, c+k} \setminus \dom(w)$, and sliding columns $\{c+1, \dotsc, n\}$ to the left. See Figure~\ref{fig:deletecols}
for a comparison between deleting a special column
and deleting a solid column. Set $w' = \del_{*,c}(w)$, and let $E = \Phi(w) \setminus (\iota_{*,c}(\Phi(w')) \cup \Phi(w)_{*,c})$. 
\begin{figure}
    \centering
    \includegraphics[width=0.8\linewidth]{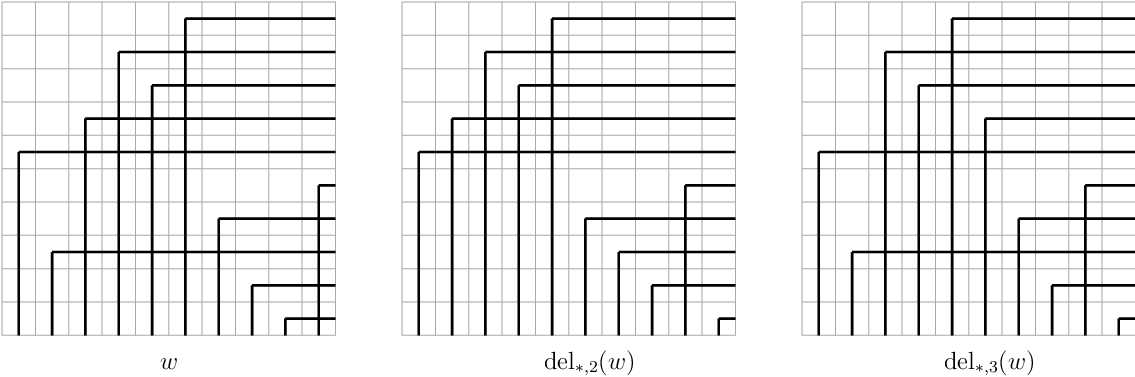}
    \caption{For the given $w$, $\del_{*,2}(w)$ is the deletion of a special column, and $\del_{*,3}(w)$ is the deletion of a solid column.}
    \label{fig:deletecols}
\end{figure}
\begin{lemma}\label{lem:solidcontaiment}
The map $\mathbb{A}^{\Phi(w') \cup E} \to \mathbb{A}^{\Phi(w)}$ induced by $\iota_{*,c}$ restricts to an isomorphism from $X_{w'} \cap \mathbb{A}^{\Phi(w') \cup E}$ to $X_w \cap \mathbb{A}^{\Phi(w)} \cap V(a_{i, c} : (i, c) \in \Phi(w))$.
\end{lemma}

\begin{proof}
After setting $a_{i, c} = 0$ for $(i, c) \in \Phi(w)$, the Fulton generators for $w$ either vanish or become Fulton generators for $w'$, and all Fulton generators for $w'$ are obtained in this way. 
\end{proof}

\begin{proposition}\label{prop:deletesolid}
Suppose that $S \subset \Phi(w)$ satisfies
$D(w)_{*,c}\subset S$,
$|S\cap \Phi(w)_{*,c}|=|D(w)_{*,c}|+1$, and $|S| = \ell(w)$.
Then 
$$\llbracket z^S\rrbracket \mathfrak{F}_{w}(z) \ge \llbracket z^{\iota_{*,c}^{-1}(S)}\rrbracket \mathfrak{F}_{w'}(z).$$
\end{proposition}

As before, Proposition~\ref{prop:deletesolid} follows from the following lemma.

\begin{lemma}
Suppose that $S \subset \Phi(w)$ satisfies
$D(w)_{*,c}\subset S$ and
$|S\cap \Phi(w)_{*,c}|=|D(w)_{*,c}|+1$. Then, under the identification induced by $\iota_{*,c}$, $X_{w'} \cap \mathbb{A}^{\Phi(w') \cup E}$ is an irreducible component of 
$$X_w \cap \mathbb{A}^{\Phi(w)} \cap \bigcap_{\substack{i, \, (i, c) \not \in S \\ (i, c) \in \Phi(w)}}V(a_{i,c}).$$
\end{lemma}

\begin{proof}
For a column $d$ and a permutation $v$, let $f_d(v)$ be the least $i$ such that $(i, d) \not \in \dom(v)$. Let $k$ be the least integer such that the $(c + k)$th column is not solid. 
We prove the lemma by induction first on $k$,  and then on $f_{c}(w) - f_{c+1}(w)$. 

First suppose that $f_{c}(w) = f_{c+1}(w)$. The solid column condition implies that this happens if and only if $w^{-1}(c) < w^{-1}(c + 1)$.
Then, by Proposition~\ref{prop:swap}, $X_w \cdot s_c = X_w$. If $k=1$, then the result follows from Lemma~\ref{lem:coliscomp}. If $k > 1$, then the result follows from induction on $k$. 

Now suppose that $f_c(w) > f_{c+1}(w)$. 
Let $\widetilde{w} = w t_{f_c(w) - 1, w^{-1}(c)}$. Then $c$ is still a solid column for $\widetilde{w}$, with $f_c(\widetilde{w}) - f_{c+1}(\widetilde{w}) = f_c(w) - f_{c+1}(w) - 1$. If we delete column $c$ from $\widetilde{w}$, then we still get $w'$.

By induction on $f_c(w) - f_{c+1}(w)$, we know that $X_{w'} \cap \mathbb{A}^{\Phi(w') \cup E}$ is an irreducible component of 
$$X_{\widetilde{w}} \cap V(a_{f_c(w)-1,c}) \cap \mathbb{A}^{\Phi(\widetilde{w})}  \cap \bigcap_{i, \, (i, c) \not \in S}V(a_{i,c}).$$
By Proposition~\ref{prop:geometriccotransition}, $X_{\widetilde{w}} \cap V(a_{f_c(w)-1,c}) = X_{w}$  because $L_{f_c(w)-1}(\widetilde{w}) = \{w\}$. As $\mathbb{A}^{\Phi(\widetilde{w})} = \mathbb{A}^{\Phi(w)}$, this gives the result. 
\end{proof}

We say that a row $r$ is a \textbf{solid row}  of $w$ if column $r$ is a solid column of $w^{-1}$. In this way, we can use Proposition~\ref{prop:deletesolid} to delete solid rows.

\subsection{Transition}
\label{subsec:transition}
In this section, we prove a version of Lascoux's transition formula \cite{Lascoux}.

\begin{proposition}\label{prop:transition}
Let $w \in S_n$, and let $(r, c) \in \mathrm{ess}(w)$ such that for all $(r',c')\in D(w)$ with $(r', c') \not= (r, c)$, either $r'<r$ or $c'<c$. Let $w':=wt_{r,w^{-1}(c)}$. 
Then for any $S$ which contains $(r, c)$, we have
$$\llbracket z^S\rrbracket \mathfrak{F}_{w}(z) = \llbracket z^{S \setminus \{(r, c)\}}\rrbracket \mathfrak{F}_{w'}(z).$$
\end{proposition}
The permutation $w'$ satisfies
the following:
\begin{enumerate}
 \item $\ell(w') = \ell(w) - 1$,
 \item $(r, c) \in \Phi(w)$, but $(r, c) \not \in \Phi(w')$, and
 \item For every $(i,j) \in D(w')$, we have $r_{i \times j}(w') = r_{i \times j}(w)$. 
\end{enumerate}

\begin{proof}[Proof of Proposition~\ref{prop:transition}]
Let $\pi \colon (\mathbb{P}^1)^{n^2} \to (\mathbb{P}^1)^{n^2 - 1}$ be the projection away from the factor labeled by $(r, c)$. We claim that $\pi(\overline{X}_w) = \pi(\overline{X}_{w'})$. Note that, because $(r, c) \not \in \Phi(w')$, $\overline{X}_{w'} = \pi(\overline{X}_{w'}) \times \mathbb{P}^1$, and so $\pi(\overline{X}_{w'})$ is an irreducible variety of dimension $n^2 - \ell(w') - 1 = n^2 - \ell(w)$. 

Also, $\pi(\overline{X}_w)$ is an irreducible variety of dimension $n^2 - \ell(w)$. That it is irreducible and has dimension at most $n^2 - \ell(w)$ is automatic; that it has dimension $n^2 - \ell(w)$ follows from the fact that $(r, c) \in \Phi(w)$. 

As $\pi(\overline{X}_w)$ and $\pi(\overline{X}_{w'})$ are irreducible varieties of the same dimension, it suffices to show that $\pi(X_{w'}) \supset \pi(X_w)$. 
Because $(r, c) \not \in \Phi(w')$, it is clear that the variety $\pi(X_{w'})$ is cut out by the Fulton generators for $I_{w'}$. Each Fulton generator for $I_{w'}$ is a Fulton generator for $I_w$, which implies the containment.

Note that the map $\overline{X}_w \to \pi(\overline{X}_w)$ is birational: over some dense open in $X_w$, any CDG generator whose spread contains $(r, c)$ allows us to solve for $a_{r,c}$ in terms of the other variables. 

We have 
$$[\pi(\overline{X}_{w'})] = \mathfrak{F}_{w'}([H_{1,1}], \dotsc, [H_{n,n}]) \smallfrown [(\mathbb{P}^1)^{n^2 - 1}] \in H_*((\mathbb{P}^1)^{n^2 - 1}).$$
Note that this makes sense, because $\mathfrak{F}_{w'}(z)$ does not involve the variable $z_{r,c}$ as $(r, c) \not \in \Phi(w')$. On the other hand,
$$[\pi(\overline{X}_w)] = \pi_*[\overline{X}_w] = \sum_{S \ni (r,c)} \llbracket z^{S}\rrbracket \mathfrak{F}_{w}(z) \cdot [H]^{S \setminus \{(r, c)\}} \smallfrown [(\mathbb{P}^1)^{n^2 - 1}].$$
Here $[H]^T := \prod_{(i, j) \in T} [H_{i,j}]$, and the first equality holds because the degree of the map $\overline{X}_w \to \pi(\overline{X}_w)$ is $1$. This gives the result. 
\end{proof}
Note that, unlike all other results in this section, Proposition~\ref{prop:transition} exactly computes some coefficients of the fine Schubert polynomial of $w$ in terms of the fine Schubert polynomial of $w'$.

\subsection{Relation to Schubert polynomials}

We now prove Proposition~\ref{prop:coeffbound}. Let $(\mathbb{P}^n)^n$ be the projective completion of the rows of the space of $n \times n$ matrices, so there is an inclusion of $\mathbb{A}^{n^2}$ into $(\mathbb{P}^n)^n$. For $X \subset \mathbb{A}^{n^2}$, let $\operatorname{cl}(X)$ denote the closure in $(\mathbb{P}^n)^n$, and let $\overline{X}$ be the closure in $(\mathbb{P}^1)^{n^2}$.  Let $[H_i] \in H^2((\mathbb{P}^n)^n)$ be the class of a hyperplane in the $i$th row. 

\begin{lemma}\label{lem:classYw}
We have $[\operatorname{cl}(X_w)] = \mathfrak{S}_w([H_1], \dotsc, [H_n]) \smallfrown [(\mathbb{P}^n)^n]$ in the homology of $(\mathbb{P}^n)^n$.
\end{lemma}

\begin{proof}
Let $(\mathbb{P}^{n-1})^n$ be the projectivization of the rows of the space of $n \times n$ matrices. 
By \cite{KM05} (see \cite[Notes to Chapter 8]{MillerSturmfels}), the class of the projectivization of the rows of $X_w$ in  $(\mathbb{P}^{n-1})^n$ is given by $\mathfrak{S}_w(x)$ evaluated at the hyperplane classes. As $\operatorname{cl}(X_w)$ is the cone over the projectivization of the rows of $X_w$, this implies the result. 
\end{proof}

\begin{proof}[Proof of Proposition~\ref{prop:coeffbound}]
Let $U = S^c$. For each $(i, j) \in U$, let $\lambda_{i,j}$ be a generic complex number, and let $V_{i,j} = \{a_{i,j} = \lambda_{i,j}\} \subset \mathbb{A}^{n^2}$.  The fundamental class of $\operatorname{cl}({V}_{i,j})$ is $[H_i]$.

As the $\lambda_{i,j}$ are generic, a version of Bertini's theorem \cite[Lemma B.9.1]{Fulton} implies that $\overline{X}_w \cap \bigcap_{(i, j) \in U} \overline{V}_{i,j}$ consists of $\llbracket z^S \rrbracket \mathfrak{F}_{w}(z)$-many reduced points which are contained in $\mathbb{A}^{n^2}$. 
 By Lemma~\ref{lem:classYw}, we have
$$ (\prod_{(i, j) \in U} [\operatorname{cl}({V}_{i,j})]) \smallfrown [\operatorname{cl}(X_w)]   = (\mathfrak{S}_w([H_1], \dotsc, [H_n]) \smallsmile \prod_{(i, j) \in U} [H_i])\smallfrown [(\mathbb{P}^n)^n] = \llbracket \prod_{(i, j) \in S}x_i \rrbracket \mathfrak{S}_w(x) \cdot [\operatorname{pt}],$$
where $[\operatorname{pt}]$ is the class of a point. 

Each of the $\llbracket z^S \rrbracket \mathfrak{F}_{w}(z)$-many reduced points is an irreducible component of $\operatorname{cl}(X_w) \cap \bigcap_{(i, j) \in U} \operatorname{cl}({V}_{i,j})$,  and these points contribute $\llbracket z^S \rrbracket \mathfrak{F}_{w}(z) \cdot [\operatorname{pt}]$ to the intersection number. Because the tangent bundle of $(\mathbb{P}^n)^n$ is globally generated,  the other irreducible components of $\operatorname{cl}(X_w) \cap \bigcap_{(i, j) \in U} \operatorname{cl}(V_{i,j})$ contribute non-negatively, which gives the result. 
\end{proof}

\begin{remark}
The above argument can be used to give a positive answer to a special case of \cite[Question 7.4]{CDGsurvey}. However, the question has a negative answer in general: a positive answer would imply that if all coefficients of $\mathfrak{S}_w(x)$ are 0 or 1, then all coefficients of the double Schubert polynomial $\mathfrak{S}_w(x, y)$ are 0 or 1. This fails for $w = 2143$.
\end{remark}

\begin{remark}
Remark~\ref{rem:CS} implies that if $\mathfrak{F}_w(z)$ has all coefficients equal to $0$ or $1$, then $I_w$ has a squarefree-supported universal Gr\"{o}bner basis. In particular, Proposition~\ref{prop:coeffbound} implies that this holds if all coefficients of $\mathfrak{S}_w(x)$ or $\mathfrak{S}_{w^{-1}}(x)$ are equal to $0$ or $1$, generalizing part of \cite[Theorem 4.6]{CDGsurvey}. However, not all $w$ for which the coefficients of $\mathfrak{F}_w(z)$ are all equal to $0$ or $1$ satisfy this. For example, the interesting part of $X_w$ could be a product of the interesting part of the matrix Schubert varieties of two permutations $u$ and $v$, where all coefficients of $\mathfrak{S}_u(x)$ are equal to $0$ or $1$ and all coefficients of $\mathfrak{S}_{v^{-1}}(x)$ are equal to $0$ or $1$. A more subtle example is $w = 46817253$, which a Macaulay2 computation shows has all coefficients of $\mathfrak{F}_w(z)$ equal to $0$ or $1$. It would be interesting to classify permutations $w$ such that $I_w$ has a squarefree-supported universal Gr\"{o}bner basis. 
\end{remark}

\section{Description of $w$ when $\mathfrak{S}_w(x)$ and
$\mathfrak{S}_{w^{-1}}(x)$ have all coefficients equal to 0 or 1}\label{sec:patterns}
\subsection{Rothe diagrams}
In this section, we provide a characterization of the permutations $w$ for which $\mathfrak{S}_w(x)$ and $\mathfrak{S}_{w^{-1}}(x)$ have all coefficients equal to $0$ or $1$ in terms of Rothe diagrams. Our proof of Theorem~\ref{thm:main} will heavily rely on  the structure of these permutations.

\begin{proposition}\label{prop:patterncriterion}
    For $w\in S_n$, $\mathfrak{S}_w(x)$ and $\mathfrak{S}_{w^{-1}}(x)$ have all coefficients equal to $0$ or $1$ if and only if $w$ avoids 12543, 13254, 13524, 14253, 13542, 15243,  21543, 125634, 215634, 315624, 251634, 315642, and 261534.
\end{proposition}
We denote this set of patterns by $\mathcal{P}$. See Figure~\ref{fig:allperms}.
\begin{figure}
    \centering
    \includegraphics[width=0.8\linewidth]{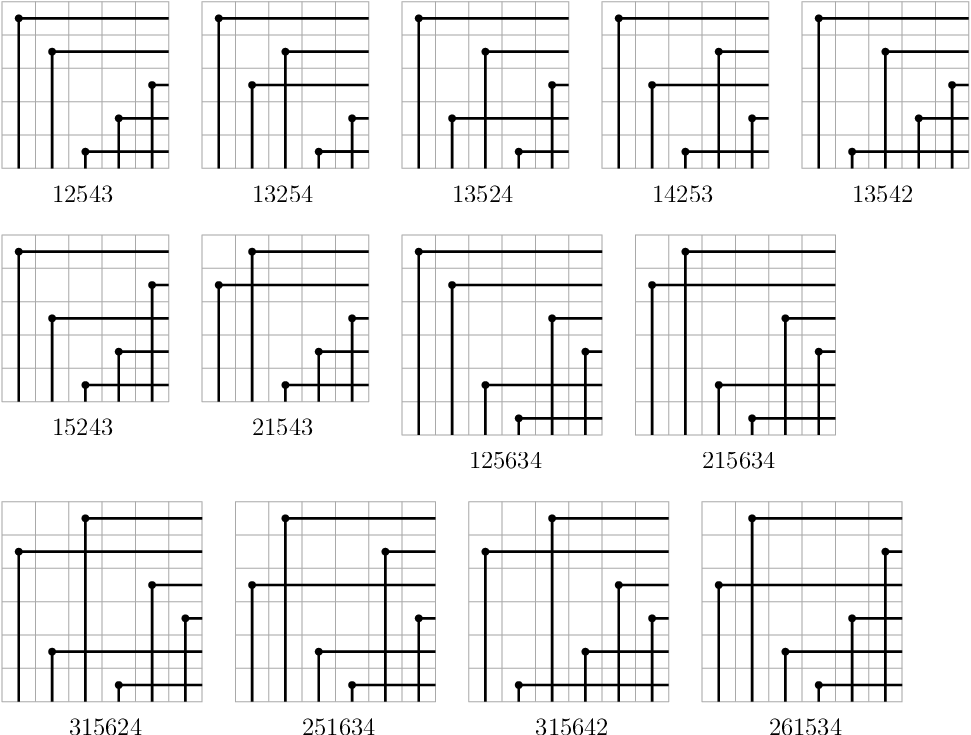}
    \caption{Patterns in $\mathcal{P}$}
    \label{fig:allperms}
\end{figure}
\begin{proof}
    By \cite[Theorem 4.8]{FMS}, $\mathfrak{S}_w(x)$ has all coefficients equal to $0$ or $1$ if and only if $w$ avoids 12543, 13254, 13524, 13542, 21543, 125364, 125634, 215364, 215634, 315264, 315624, and 315642. The set $\mathcal{P}$ is obtained by adding the inverses of these patterns and removing the redundant ones. 
\end{proof}

\begin{figure}
    \centering
    \includegraphics[width=0.55\linewidth]{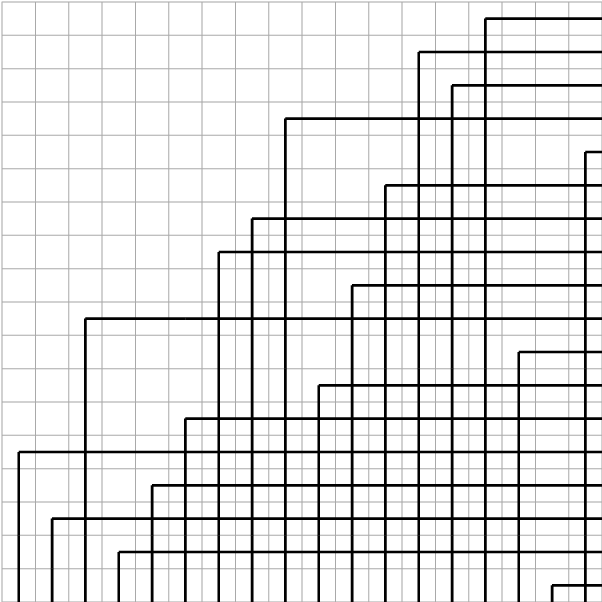}
    \caption{A large permutation which avoids $\mathcal{P}$}
    \label{fig:largeperm}
\end{figure}

We say that $(p,w(p))$ is an \textbf{outer corner} of $\dom(w)$ if $(p-1,w(p))\in \dom(w)$ and $(p,w(p)-1)\in \dom(w)$, whenever they are defined. Note that if $\dom(w)=\emptyset$, then necessarily $w(1)=1$ and $(1,1)$ is an outer corner of $\dom(w)$. We say that $(i,j)\in D(w)$
is a \textbf{northwest corner} if $(i-1,j)\not\in D(w)$
and $(i,j-1)\not\in D(w)$.

\begin{lemma}
    \label{lem:rank-1-diagram}
    Suppose $w$ avoids $\mathcal{P}$, $(p,w(p))$ is an outer corner of $\dom(w)$, and $(p+1,w(p)+1)\in D(w)$.
    Let $D_p:=\{(i,j)\in D(w): i> p, j> w(p)\}$, and let $D_p'\subset D_p$ be the region connected to $(p+1, w(p)+1)$. Then if $(i,j)\in D_p\setminus D_p'$, then either $i=p+1$ or $j=w(p)+1$.
\end{lemma}
\begin{proof}
Let $a$ be the largest row index such that $(a,w(p)+1)\in D_p'$, and let $b$ be the largest column index such that $(p+1,b)\in D_p'$. We will use the following elementary observation. Since $(p,w(p))$
is an outer corner of $\dom(w)$ and $(p+1,w(p)+1)\in D(w)$,
the part of the component $D_p'$ meeting the first column
$w(p)+1$ contains the whole vertical segment
$[p+1,a]\times\{w(p)+1\},$
and the part meeting the first row contains the whole horizontal
segment
$\{p+1\}\times [w(p)+1,b].$
Moreover, for every $p+1\le r\le a$ and every $t<w(p)$, one has
$(r,t)\in D(w)$; the transpose statement holds for columns
$w(p)+1\le c\le b$.

    First suppose that there is a region of $D_p \setminus D_p'$ which has a northwest corner that is contained in the $(p+1)$st row, but the region is not contained in row $p+1$. Let $(p+1, j)$ be such a northwest corner with $j$ minimal. Then $(p+2,j)\in D_p$, so $w(p+2)>j$, and necessarily $(p+2,w(p)+1)\in D_p$. Notice also that, by the minimality of $j$,  for all $b<k<j$ we have $(p+1,k)\not\in D(w)$. Therefore, for all $b<j'<j$, we have $w^{-1}(j')<p$.
    \begin{itemize}
        \item If $w(p+1)>w(p+2)$, we have $w^{-1}(j-1)<p<p+1<p+2<w^{-1}(j)$ and $w(p)<j-1<j<w(p+2)<w(p+1)$, realizing the pattern 21543. 
        \item If $w(p+1)<w(p+2)$:
        \begin{itemize}
            \item If $w^{-1}(w(p)+1)<w^{-1}(j)$, we have $w^{-1}(j-1)<p<p+1<p+2<w^{-1}(w(p)+1)<w^{-1}(j)$ and $w(p)<w(p)+1<j-1<j<w(p+1)<w(p+2)$, so we realize the pattern
            315624.
            \item If $w^{-1}(w(p)+1)>w^{-1}(j)$, we have $w^{-1}(j-1)<p<p+1<p+2<w^{-1}(j)<w^{-1}(w(p)+1)$ and $w(p)<w(p)+1<j-1<j<w(p+1)<w(p+2)$, so we realize the pattern 315642.
        \end{itemize}
    \end{itemize}
    
    If there is a region of $D_p \setminus D_p'$ which has a northwest corner that is contained in the $(w(p) + 1)$st column, but the region is not contained in column $w(p) + 1$, then we are in the situation transposed from the previous case. We can therefore realize the patterns 21543, 251634, or 261534.

    Finally, suppose $(i,j)$ is a northwest corner of one of the regions of
$D_p\setminus D_p'$ with $i>p+1$ and $j>w(p)+1$, and there is no such northwest
corner $(i_0,j_0)\neq(i,j)$ with $i_0\le i$ and $j_0\le j$. We say $(i,j)$ is a northwesternmost northwest corner.
    Then $w(i-1)<j$, $w(i)>j$, $w^{-1}(j-1)<i$, and $w^{-1}(j)>i$. 
    There are the following cases:
   
         \underline{Case $i\le a$}. By assumption, for all $p+1\le i' \le a$ and $j'<w(p)$, we have $(i',j')\in D(w)$ and $(i',w(p)+1)\in D(w)$. Therefore, since $p+1\le i-1\le a$, we must have $w(i-1)>w(p)+1$. 
        Furthermore, we have $w^{-1}(w(p)+1)>a$.
        
        Then if $w^{-1}(j)<w^{-1}(w(p)+1)$, we have $p<i-1<i<w^{-1}(j)<w^{-1}(w(p)+1)$
        and $w(p)<w(p)+1<w(i-1)<j<w(i)$, realizing the pattern 13542. Otherwise  $p<i-1<i<w^{-1}(w(p)+1)<w^{-1}(j)$ and $w(p)<w(p)+1<w(i-1)<j<w(i)$, realizing the pattern 13524.

        \underline{Case $j\le b$.} This is transpose to the case $i\le a$, and we can realize patterns 15243 and 14253.

        \underline{Case $i>a+1$ and $j>b+1$.} (Notice that it is impossible that $i=a+1$ or $j=b+1$). There are a few subcases.
   
        \underline{Subcase 1:} $w(i)<w(p+1)$ and for all $b<j'<j$, we have $w^{-1}(j')<p$. It must be the case that $(p+1,j)\in D(w)$, so $w(p+1)>j$. We have $w^{-1}(j-1)<p<p+1<i<w^{-1}(j)$ and $w(p)<j-1<j<w(i)<w(p+1)$, realizing the pattern 21543.  

        \underline{Subcase 2:} $w^{-1}(w(p)+1)>w^{-1}(j)$, and for all $a<i'<i$, we have $w(i')<w(p)$. We are in the situation transpose to Subcase 1, and we can again realize 21543 with row indices $p<i-1<i<w^{-1}(j)<w^{-1}(w(p)+1)$.

        \underline{Subcase 3:} $w(i)>w(p+1)$; $w^{-1}(w(p)+1)<w^{-1}(j)$; for all $b<j'<j$, we have $w^{-1}(j')<p$; and for all $a<i'<i$ we have $w(i')<w(p)$. Then it must be the case that $(i,w(p)+1)\in D(w)$, so $w^{-1}(w(p)+1)>i$. We have $p<p+1<i-1<i<w^{-1}(w(p)+1)<w^{-1}(j)$ and $w(i-1)<w(p)<w(p)+1<j<w(p+1)<w(i)$, realizing the pattern 251634. 

        \underline{Subcase 4:} $w(i)>w(p+1)$; for all $b<j'<j$, we have $w^{-1}(j')<p$; and there exists $a<i'<i$ such that $w(i')>w(p)$. Since $(i,j)\in D(w)$ is a northwest corner, we must have $w(i')<j$. Again we have $(p+1,j)\in D(w)$. We have $p<p+1<i'<i<w^{-1}(j)$ and $w(p)<w(i')<j<w(p+1)<w(i)$, realizing the pattern 14253.

        Notice that Subcases 1--4 cover all the situations when, for all $b<j'<j$, we have $w^{-1}(j')<p$. By considering the transpose, we cover all the situations when, for all $a<i'<i$, we have $w(i')<w(p)$. We are left with one last case.

        \underline{Subcase 5:} There exists $b<j'<j$ such that $w^{-1}(j')>p$ and there exists $a<i'<i$ such that $w(i')>w(p)$. Since $(i,j)\in D(w)$ is a northwesternmost northwest corner, we must have $w(i')<j$ and $w^{-1}(j')<i$.

        First we assume that $i',j'$  also satisfy $w^{-1}(j')\le a $  and $w(i')\le b$. We then have $p<w^{-1}(j')<i'<i<w^{-1}(j)$ and $w(p)<w(i')<j'<j<w(i)$, realizing the pattern 13254.

        We then consider the case when for all $a<\tilde{i}<i$ such that $w(\tilde{i})>w(p)$, we also have $w(\tilde{i})>b$. It follows that $w^{-1}(w(p)+1)>i$, since $w^{-1}(w(p)+1)>a$.  If $w^{-1}(w(p)+1)<w^{-1}(j)$, we have $p<i'<i<w^{-1}(w(p)+1)<w^{-1}(j)$ and $w(p)<w(p)+1<w(i')<j<w(i)$, realizing the pattern 13524. If $w^{-1}(w(p)+1)>w^{-1}(j)$, we realize the pattern 13542.

        The case when, for all $b<\tilde{j}<j$ such that $w^{-1}(\tilde{j})>p$, we also have $w^{-1}(\tilde{j})>a$, is transpose to the previous case.
\end{proof}

\begin{lemma}
    \label{lem:high-rank-diagram}
    Suppose $w$ avoids $\mathcal{P}$, $(p,q)\in D(w)$ is the northwest corner of a non-dominant region, and either (a) $w(p-1)\neq q-1$, or (b) $q\ge 3$,  $w(p-1)=q-1$ and $w^{-1}(q-2)\le p-2$. Let $D_{p,q}:=\{(i,j)\in D(w): i\ge p, j\ge q\}$. Then $D_{p,q}$ is either entirely contained in row $p$ or entirely contained in column $q$. Furthermore, if $(s,t)\in D(w)$ and both $s<p$ and $t<q$, then $(s,t)\in \dom(w)$. 
\end{lemma}

\begin{proof}
    Since $(p,q)\in D(w)$ is a northwest corner of a non-dominant region, when (a) holds we have $w(p-1)<q-1$ and $w^{-1}(q-1)<p-1$. 
    
    First suppose $(i,j)\in D_{p,q}$ such that $i>p$, $j>q$. 

    \underline{Case 1}: Suppose $(p,j),(i,q)\in D_{p,q}$. Then we have $w(p)>j$, $w(i)>j$, $w^{-1}(q)>i$, and $w^{-1}(j)>i$. 
    
    If $w(p)<w(i)$ and $w^{-1}(q)<w^{-1}(j)$, then when (a) holds, we have $w^{-1}(q-1)<p-1<p<i<w^{-1}(q)<w^{-1}(j)$ and $w(p-1)<q-1<q<j<w(p)<w(i)$, realizing the pattern 215634. When (b) holds, we can realize the pattern 125634.

    If $w(p)>w(i)$, then when (a) holds, we have $w^{-1}(q-1)<p-1<p<i<w^{-1}(j)$ and $w(p-1)<q-1<j<w(i)<w(p)$, realizing the pattern 21543. When (b) holds, we can realize the pattern 12543. The case $w^{-1}(q)>w^{-1}(j)$ is similar.

    \underline{Case 2}: Suppose $(i,q)\not\in D_{p,q}$ and $(p,j)\not\in D_{p,q}$. We must have some $p<i'<i$ such that $w(i')=q$, and some $q<j'<j$ such that $w(p)=j'$. We have $p-1<p<i'<i<w^{-1}(j)$ and $w(p-1)<q<w(p)<j<w(i)$, realizing the pattern 13254. 

    \underline{Case 3}: Suppose $(i,q)\not\in D_{p,q}$ and $(p,j)\in D_{p,q}$. Then we must have some $p<i'<i$ such that $w(i')=q$, as well as $w(p)>j$. If $w(p)<w(i)$, we have $p-1<p<i'<i<w^{-1}(j)$ and $w(p-1)<q<j<w(p)<w(i)$, realizing the pattern 14253. If $w(p)>w(i)$, we realize the pattern 15243. 

    \underline{Case 4}: If $(i,q)\in D_{p,q}$ and $(p,j)\not\in D_{p,q}$, we are in the situation transpose to Case 3, realizing the pattern 13524 or 13542. 

    The above argument shows that we cannot have $(i,j)\in D_{p,q}$ such that $i>p$, $j>q$. It remains to show that there cannot simultaneously exist $(p,j)$ and $(i,q)$ in $D_{p,q}$ with $j>q$ and $i>p$.

    Suppose otherwise. We have $w(p)>j$, $w^{-1}(q)>i$. Since there are no boxes strictly to the southeast of $(p,q)$, we have $(i,j)\not\in D_{p,q}$. Therefore, either there exists some $p<i'\le i$ such that $w(i')=j$, or there exists some $q<j'\le j$ such that $w(i)=j'$. Suppose the former is true; the latter case is similar. When (a) holds, 
    we have $w^{-1}(q-1)<p-1<p<i'<w^{-1}(q)$ and $w(p-1)<q-1<q<j<w(p)$, realizing the pattern 21543. When (b) holds, we realize the pattern 12543.
    Finally, suppose there exists $(s,t)\in D(w)$ with $s<p$ and $t<q$ such that  $(s,t)\not\in \dom(w)$. Without loss of generality we may assume $(s,t)$ is the northwest corner of its region, and if $(s',t')\in D(w)$ is strictly northeast of $(s,t)$, then $(s',t')\in \dom(w)$. If $(s-1, t-1)$ is an outer corner of $\dom(w)$, then by Lemma \ref{lem:rank-1-diagram}, $(p,q)$ cannot be in $D(w)$, giving a contradiction. Otherwise, $(s,t)$ satisfies conditions (a) or (b) of this lemma, so the previous argument shows that $(p,q)$ cannot be in $D(w)$.
\end{proof}

\begin{lemma}
\label{lem:patternconverse}
    If $w$ contains a pattern $v\in\mathcal{P}$,
    then $w$ does not satisfy the properties in Lemma~\ref{lem:rank-1-diagram} or Lemma~\ref{lem:high-rank-diagram}.
\end{lemma}

\begin{proof}
    By inspection, if $w\in \mathcal{P}$, then $D(w)$ violates
    the properties in Lemma~\ref{lem:rank-1-diagram}
    or Lemma~\ref{lem:high-rank-diagram}. 

    Let $v\in\mathcal{P}$,  so $v\in S_m$ ($m=5$ or 6), and let $w\in S_n$.
    Suppose we have  $k_1<\dots <k_m$, and the relative
    order of $w(k_1),\dots, w(k_m)$ is given by $v$. 
    Suppose $\{l_1,\dots l_m\}=\{w(k_1),\dots w(k_m)\}$ and $l_1<\dots <l_m$.
    Let $\iota \colon [m]^2\to [n]^2$ be given by
    $\iota((i,j)):=(k_i,l_j)$. Then 
    $(i,j)$ lies in $ D(v)$ if and only if $\iota((i,j))$ lies in $ D(w)$, and $r_{k_i\times l_j}(w)\ge r_{i\times j}(v)$.
    If $D(v)$ violates the conditions in Lemma~\ref{lem:rank-1-diagram}, then $D(w)$ violates the conditions in 
    Lemma~\ref{lem:rank-1-diagram} if the rank 1 region stays
    rank 1, and $D(w)$ violates the conditions in Lemma~\ref{lem:high-rank-diagram} if the rank 1 region increases rank. If $D(v)$ violates the conditions
    in Lemma~\ref{lem:high-rank-diagram}, then $D(w)$  also does.
\end{proof}

In summary, Lemmas~\ref{lem:rank-1-diagram}, \ref{lem:high-rank-diagram}, and \ref{lem:patternconverse} give the following.

\begin{theorem}
    \label{thm:characterization}
    A permutation $w$ avoids $\mathcal{P}$ if and only if its
    diagram $D(w)$ satisfies the
    conditions in Lemmas~\ref{lem:rank-1-diagram} and \ref{lem:high-rank-diagram}.
\end{theorem}

By considering how the various operations on a permutation in Section~\ref{sec:geom} change the
Rothe diagram, we have the following corollary.
\begin{corollary}
\label{cor:del-preserve-p}
    Suppose $w$ is a permutation that avoids $\mathcal{P}$ and $w'$ is a permutation obtained
    from $w$ by applying one of $\del_k$, $\del_{r,*}$, $\del_{*,c}$, or
    transition as in Proposition~\ref{prop:transition}.
    Then $w'$ avoids $\mathcal{P}$.
\end{corollary}

\begin{proposition}
\label{prop:cotrans-preserves-P}
    Suppose $w$ avoids $\mathcal{P}$, and let $w'\in L_k(w)$ be
    a term obtained by cotransitioning at $k$.
    Then $w'$ also avoids $\mathcal{P}$.
\end{proposition}
\begin{proof}
    By Proposition~\ref{prop:patterncriterion},
    $\mathfrak{S}_w(x)$ and $\mathfrak{S}_{w^{-1}}(x)$ have all coefficients equal to 0 or 1.
    Since $x_k\mathfrak{S}_w(x)=\sum_{w'\in L_k(w)}\mathfrak{S}_{w'}(x)$, each $\mathfrak{S}_{w'}(x)$
    must have all coefficients equal to 0 or 1.
    From \cite{KnutsonCotransition}, we also have
    $(x_{k}-y_{w(k)})\mathfrak{S}_{w}(x,y)=\sum_{w'\in L_k(w)}\mathfrak{S}_{w'}(x,y)$ for double
    Schubert polynomials. Since
    $\mathfrak{S}_w(x)=\mathfrak{S}_{w}(x,0)$
    and $\mathfrak{S}_{w}(x,y)=\mathfrak{S}_{w^{-1}}(-y,-x)$ (see e.g., \cite[Section 11.4]{AndersonFulton}), it follows that $\mathfrak{S}_{{(w')}^{-1}}(x)$
    also has all coefficients $0$ or $1$ for $w'\in L_k(w)$.
\end{proof}

\subsection{Relations in $\mathcal{R}_w$}
Equipped with the characterization of 
the Rothe diagrams of $w$, we state some (mostly straightforward) properties of
relations in $\mathcal{R}_w$ as technical preparation for the next section.

Let $f\in \mathcal{R}_w$, and let $\operatorname{cdg}(f)$
be the set of CDG generators merged to get $f$.  (This may not be unique 
in general; we may take the ones picked by
any specific implementation of the algorithm.)
Let $\mathfrak{m}(f):=\{(i,j):\exists g_1\neq g_2\in \operatorname{cdg}(f),  \text{ such that }(i,j)\in \operatorname{spr}(g_1)\cap \operatorname{spr}(g_2)\}$, the set of coordinates at which merging occurred. By construction, we have
$\operatorname{spr}(f)=(\bigcup_{g\in \operatorname{cdg}(f)}\operatorname{spr}(g)) \setminus \mathfrak{m}(f)$.

\begin{lemma}
    \label{lem:nosingleton}
    Let $f\in\mathcal{R}_w$ of degree at least 2. If $\operatorname{spr}(f)_{*,c}\neq \emptyset$,
    then
    $|\operatorname{spr}(f)_{*,c}|>1$.
\end{lemma}
\begin{proof}
    We induct on the number of CDG generators
    merged to construct $f$.
    If $f\in\mathrm{CDG}(w)$, this is clear
    by definition of the CDG generators (Section~\ref{sec:intro}).
    Suppose $f=\mathrm{merge}_{(\alpha,\beta)}(g,h)$.
    If $\beta\neq c$, then $\operatorname{spr}(f)_{*,c}=\operatorname{spr}(g)_{*,c}$ or
    $\operatorname{spr}(f)_{*,c}=\operatorname{spr}(h)_{*,c}$ and
    the claim follows by induction.
    If $\beta=c$, then $\operatorname{spr}(f)_{*,c}=\operatorname{spr}(g)_{*,c}\cup \operatorname{spr}(h)_{*,c}\setminus \{(\alpha,c)\}$. Since $\operatorname{spr}(g)_{*,c}\cap \operatorname{spr}(h)_{*,c}=\{(\alpha, c)\}$, the claim follows.
\end{proof}

\begin{lemma}
\label{lem:cotrans-column-ok}
    Suppose that $w$ avoids $\mathcal{P}$, and that
    \begin{enumerate}
        \item $(r,c)\in \mathrm{ess}(w)$ has rank $m$;
        \item there is $r'<r$ such that $w(r')=c-1$ and 
    $(r'+1,c)\in D(w)$ is rank 1;
    \item there is $c'<c$ such that $w^{-1}(c')=r-1$ and $(r,c'+1)\in D(w)$;
    \item for $i\in [r-(m-1),r-1]$, $w(i)<w(r')$.
    \end{enumerate}
    
    If $f\in\mathcal{R}_w$
    has some $g\in \operatorname{cdg}(f)$ such that $g$ is 
    obtained from an $(m+1)$-minor 
    and $\operatorname{spr}(g)\subset [r]\times [c]$, then
    $\operatorname{spr}(f)_{[r-(m-1),r],[c-1,c]}=([r-(m-1),r]\times[c-1,c])$. (See Figure~\ref{fig:caseII1} for an illustration, where $(r,c)$ here is $(r_2,c_2)$ in the figure.) 
\end{lemma}

\begin{proof}
    By the definition of the CDG generators,
    $g$ cannot use more than one row above row $r-(m-1)$, and it also must use both
    columns $c$ and $c-1$. Since it is an $(m+1)$-minor, $[r-(m-1),r]\times[c-1,c]\subset \operatorname{spr}(g)$. Since none of the coordinates in $[r-(m-1),r]\times[c-1,c]$
    can be in $\mathfrak{m}(f)$, the result follows.
\end{proof}

\section{Proof of Theorem~\ref{thm:main}}\label{sec:proof}

We first outline the strategy that we use to prove Theorem~\ref{thm:main}. Given $w \in S_n$ and a subset $U$ of $[n]^2$, we say that $U$ is \emph{$X_w$-completable} if $\pi_U(\overline{X}_w) = (\mathbb{P}^1)^U$, i.e., if we fill in the entries corresponding to $U$ in an $n \times n$ matrix with generic complex numbers, then we can fill in the entries in $[n]^2 \setminus U$ in some way so that the resulting matrix lies in $X_w$. Combining Lemma~\ref{lem:spreadprojection} and Proposition~\ref{prop:spreadmultidegree} gives that $U$ is $X_w$-completable if and only if there is $T \subset U^c$ with $|T| = \ell(w)$ such that $\llbracket z^T \rrbracket \mathfrak{F}_{w}(z) \not=0$. If $V$ is $X_w$-completable and $U \subset V$, then $U$ is $X_w$-completable. 

We say that $U$ is \emph{$\mathcal{R}_w$-avoiding} if $U$ does not contain the spread of any $f \in \mathcal{R}_w$. If $V$ is $\mathcal{R}_w$-avoiding and $U \subset V$, then $U$ is $\mathcal{R}_w$-avoiding. 
By Lemma~\ref{lem:spreadprojection}, if $U$ is $X_w$-completable, then it is $\mathcal{R}_w$-avoiding. In general, the converse does not hold. 

\begin{example}\label{ex:21543}
Let $w = 21543$. A Macaulay2 computation \cite{M2} shows that
\begin{equation*}\begin{split}
f = a_{1,2}a_{1,4}a_{2,1}a_{3,1}a_{4,3} - a_{1,3}a_{1,4}a_{2,1}a_{3,1}a_{4,2} + a_{1,2}a_{1,3}a_{2,1}a_{3,4}a_{4,1} &  \\  - a_{1,2}a_{1,4}a_{2,1}a_{3,3}a_{4,1} - a_{1,2}a_{1,3}a_{2,4}a_{3,1}a_{4,1} + a_{1,3}a_{1,4}a_{2,2}a_{3,1}a_{4,1} \in I_w.
\end{split}\end{equation*}
The spread of $f$ does not contain the spread of any relation which can be obtained by merging CDG generators, so $\operatorname{spr}(f)$ is $\mathcal{R}_w$-avoiding but not $X_w$-completable.
\end{example}

\begin{theorem}\label{thm:Rwequiv}
Let $w$ be a permutation which avoids $\mathcal{P}$. Then $U$ is $\mathcal{R}_w$-avoiding if and only if it is $X_w$-completable. 
\end{theorem}

\begin{proof}[Proof of Theorem~\ref{thm:main}]
By Proposition~\ref{prop:patterncriterion}, $w$ avoids $\mathcal{P}$ if and only if $\mathfrak{S}_w(x)$ and $\mathfrak{S}_{w^{-1}}(x)$ have all coefficients equal to $0$ or $1$. Theorem~\ref{thm:Rwequiv} implies that $\mathcal{R}_w$ satisfies the criterion in Theorem~\ref{thm:irreducible}, as the relations in $\mathcal{R}_w$ are squarefree-supported. 
\end{proof}

We now explain the strategy of the proof of Theorem~\ref{thm:Rwequiv}. The results of Section~\ref{sec:geom} describe a number of ``smaller'' permutations $w' \in S_n$ that can be constructed from $w$. We will show that if $U$ is $\mathcal{R}_w$-avoiding, then some related set $U'$ is  $\mathcal{R}_{w'}$-avoiding. Corollary~\ref{cor:del-preserve-p} and Proposition~\ref{prop:cotrans-preserves-P} show that $w'$ also avoids $\mathcal{P}$. As $w'$ is ``smaller'' than $w$, we deduce by induction that $U'$ is $X_{w'}$-completable. By Lemma~\ref{lem:algebraicmatroid}, there is some $\widetilde{U}'$ containing $U'$ such that $|\widetilde{U}'| = n^2 - \ell(w')$ and $\widetilde{U}'$ is still $X_{w'}$-completable. Then the results in Section~\ref{sec:geom} construct some $\widetilde{U}$ with $|\widetilde{U}| = n^2 - \ell(w)$ that is $X_w$-completable. The set $\widetilde{U}$ will contain $U$, so $U$ is $X_w$-completable. 

We start with a simple but important lemma.
\begin{lemma}
    \label{lem:permute-row-col}
    Let $U\subset [n]^2$ and $w\in S_n$.
    Suppose $[r,r']\subset [n]$ has the property that, for all $r\le i<r'$, $w(i)<w(i+1)$. If $v\in S_n$ fixes $[n]\setminus [r,r']$,  then $U$ is $\mathcal{R}_w$-avoiding (resp., $X_w$-completable) if and only if $v\cdot U$ is $\mathcal{R}_w$-avoiding (resp., $X_w$-completable).
\end{lemma}
We will often use the transposed version of this lemma, where we replace $w$ with $w^{-1}$ and $v\cdot U$ with $U\cdot v$ in the statement.

\begin{proof}
This follows from Proposition~\ref{prop:swap} and the construction of $\mathcal{R}_w$.
\end{proof}

\begin{proof}[Proof of Theorem~\ref{thm:Rwequiv}]
    It suffices to show the forward direction. We proceed by induction first on the size of the smallest rectangle containing $\Phi(w)$, then on the size of $\Phi(w)\setminus\dom(w)$. 
    Our base case is when $w$ is the identity permutation.
    Suppose $U\subset[n]^2$ is $\mathcal{R}_w$-avoiding. We may assume $U$ contains all entries outside of $\Phi(w)$, since these are not contained in the spread of any element of $\mathcal{R}_w$.

    If $D(w)_{*,[w(1),n]}=\emptyset$,
    let $w':=\del_{1}(w)$ and $U':=\iota_1^{-1}(U\setminus (U_{*,w(1)} \cup U_{1,*}))$. Since 
    $U$ is $\mathcal{R}_w$-avoiding, 
    $U'$ must be $\mathcal{R}_{w'}$-avoiding,
    and by the induction hypothesis it is
    $X_{w'}$-completable. Therefore, by Proposition~\ref{prop:patterncontainment}
    and the discussion at the beginning of
    this section,
    $U$ is $X_w$-completable.
    
    Otherwise, there exists $c>w(1)$ such that $D(w)_{*,c}\neq \emptyset $. Let $c$ be the largest such column index.  Let $(r,c)$ be the essential box with $r$ smallest.
    
    \underline{\textbf{Case I}: $(r,c)$ is not rank 1.} By assumption, $c>w(1)$. By Lemma~\ref{lem:high-rank-diagram}, we know that either $D(w)_{*,[w(1),n]}\subset \{r\}\times [w(1)+1,c]$, or that $D(w)_{*,[w(1),n]}$ lies in column $c$.  

\underline{\textbf{Case I.1}}: Suppose $D(w)_{*,[w(1),n]}\subset \{r\}\times [w(1)+1,c]$. Then $(r,c)$ is the unique essential box in column $c$, and  $w^{-1}$ is increasing in $[w(1),c]$.   
    Suppose there is $j\in [w(1), c]$ such that $|U_{*,j}|\le n-1$ and that $(i,j)\in \Phi(w)\setminus U$ for some $i$. 
    If $|U_{*,j}|<n-1$,
    we may enlarge $U$ so that it
    contains everything except $(i,j)$ in column $j$, since if $f\in \mathcal{R}_w$ satisfies $\operatorname{spr}(f)_{*,j}\neq \emptyset$, it must be the case that 
    $\operatorname{spr}(f)_{*,j}=[r]\times\{j\}$: indeed, this is
    apparent when $f$ is a CDG generator, and if $(i,j)$ is contained
    in the spreads of two CDG generators, they both must use all of rows 1 through $r$, so these generators cannot be merged. 
    We may also assume without loss of generality that $j=c$, since
    by Lemma~\ref{lem:permute-row-col} we may replace $U$ with
    $U\cdot t_{j,c}$ if $j\neq c$.

    Let $w':=\del_{*,c}(w)$, and let $U':=\iota^{-1}_{*,c}(U\setminus U_{*,c})$. Since $\mathcal{R}_{w'}\subset \mathcal{R}_w$, $U'$ is $\mathcal{R}_{w'}$-avoiding. By the induction hypothesis,
    $U'$ is $X_{w'}$-completable. 
    By Proposition~\ref{prop:deletearbitrary} or Proposition~\ref{prop:deleteoneoff} (depending on whether $i=r$), $U$ is $X_w$-completable.

    Now, suppose that for all $j\in [w(1),c]$, we have $|U_{*,j}|=n$. Then 
    let $w':=\del_1(w)$, and let 
    $U'\coloneqq \iota^{-1}_1(U\setminus (U_{*,w(1)} \cup U_{1,*})).$
    Suppose to the contrary that there exists some $f\in \mathcal{R}_{w'}$
    such that $\operatorname{spr}(f)\subset U'$. If $\operatorname{spr}(f)\cap ([n]\times [w(1),c])=\emptyset$, the construction of $w'$ implies that $f \in \mathcal{R}_w$. Otherwise, there must exist
    a unique $g\in \mathrm{CDG}(w')$ such that either $g=f$,
    or $f=\mathrm{merge}_{(r,\beta)}(g,h)$
    for some $h\in \mathcal{R}_{w'}$ and $\beta<w(1)$. 
    Let $\widetilde{g}$ be the unique generator in 
    $\mathrm{CDG}(w)$ such that, after evaluating $\widetilde{g}$ at $a_{1,w(1)}=1$,
    $a_{1,j}=0$ for all $w(1)<j\le c$, and $a_{i,w(1)}=0$
    for all $1<i \le r$, we get $g$. 
    Then set $\widetilde{f}:=\mathrm{merge}_{(r,\beta)}(\widetilde{g}, h)$ if $h$ exists, and set 
    $\widetilde{f}:=\widetilde{g}$ otherwise; we must have $\operatorname{spr}(\widetilde{f})\subset U$ and $\widetilde{f}\in \mathcal{R}_w$. This contradicts that $U$ is $\mathcal{R}_w$-avoiding. Therefore, $U'$ is
    $\mathcal{R}_{w'}$-avoiding, and by the induction hypothesis, it is $X_{w'}$-completable. By Proposition~\ref{prop:patterncontainment}, $U$
    is $X_w$-completable.
    
    \underline{\textbf{Case I.2}}: $D(w)_{*,[w(1),n]}$ lies in column $c$. Then if $(r',c)$ is the northwest corner of the region connected to $(r,c)$,  all boxes in $D(w)$ strictly northwest of $(r',c)$ are in $\dom(w)$. Notice that $w$ is increasing on $[r',r]$, rows in $[r',r]$ are special rows, and rows in $[r'-1]$ are solid rows. Since $U$ is $\mathcal{R}_w$-avoiding and the CDG generators given by the essential box $(r,c)$ cannot be merged, there must be some $i\in [r]$ such that $|U\cap \Phi(w)_{i,*}|\le |(\Phi(w)\setminus \dom(w))_{i,*}|-1$. If the inequality is strict, we may enlarge $U$ so that it achieves the equality. By the row version of \Cref{prop:deletearbitrary}, \Cref{prop:deleteoneoff}, or \Cref{prop:deletesolid}, we may apply induction on $w':= \del_{i,*}(w)$.

    \underline{\textbf{Case II:} $(r,c)$ is rank 1.}
    We reduce to the case when $(r,c)$ is the unique
    essential box in its region in $D(w)$, i.e., this region is rectangular.
    To do this, let $r'$ be the largest row index such
    that $(r', w(1)+1)$ is in $D(w)$ and in the same region
    as $(r,c)$. If this region is not rectangular,
    there is some $r<\rho \le r'$ such that
    $w(1)<w(\rho)\le c$, $(\rho-1, w(\rho))$ lies in $D(w)$,
    and $(\rho, w(\rho)-1)$ lies in $D(w)$. Let $\rho'>\rho$
    be the smallest such that $w(\rho')>w(\rho)$. (Note that $\rho'$ always exists if we consider $w$ in a large enough symmetric group. This does not change the matrix Schubert variety up to affine factors.)
    Then set $\widetilde{w}:=wt_{\rho,\rho'}$, which is a cover of $w$. By the characterization in Lemma~\ref{lem:rank-1-diagram} and Lemma~\ref{lem:high-rank-diagram} it is easily seen that $\widetilde{w}$ avoids $\mathcal{P}$. Since there are no cells
    in $D(w)$ weakly southeast of $(\rho, w(\rho))$, we have
    $\operatorname{CDG}(w)\subset \operatorname{CDG}(\widetilde{w})$.
    Then    $U\setminus \{(\rho, w(\rho))\}$ is $\mathcal{R}_{\widetilde{w}}$-avoiding.
    By Proposition~\ref{prop:transition}, 
    $U\setminus \{(\rho, w(\rho))\}$ is $X_{\widetilde{w}}$-completable \emph{if and only if} $U$ is 
    $X_{w}$-completable, so we may replace $w$ with $\widetilde{w}$. (We note that we do not yet apply the inductive hypothesis here, as this operation can increase the size of $\Phi(w) \setminus \dom(w)$.) 
    \smallskip

    Because $(r,c)$ is an essential box in a rectangular region, $w$ is increasing on $[r]$. Notice that
    $X_w\cap \mathbb{A}^{[1,r]\times [w(1),c]}$
    is the variety of rank at most $1$ matrices. By the discussion in Section~\ref{ssec:determinantal}, a subset $V\subset [1,r]\times [w(1),c]$ can be identified with a subgraph of a complete bipartite graph with vertex set $[1,r] \sqcup [w(1), c]$. If $V$ is
    $\mathcal{R}_w$-avoiding, this subgraph must be a forest. To do induction, we first try to delete a ``row leaf'' or a ``column leaf,'' i.e., a row or column which has only one entry in $U \cap \Phi(w)$.

    By Lemma~\ref{lem:rank-1-diagram} and the assumption on 
    $(r,c)$, we see that for all $1<i\le r$, all but one
    of the entries in $\Phi(w)_{i,*}$ are in $D(w)$, and for all $w(1)+2\le j \le c$, all but one of the entries $\Phi(w)_{*,j}$ are in $D(w)$.

    Suppose there is $i \le r$ such that $|U_{i,*}\cap \Phi(w)|\le 1$.
    If $U_{i,*}\cap \Phi(w)=\emptyset$, we may pick some $(i,j)\in \Phi(w)\setminus \dom(w)$
    to add to $U$; the result will still be $\mathcal{R}_w$-avoiding by the row version of Lemma~\ref{lem:nosingleton}.
    Suppose $(i,j)\in U$.
    If $i<r$, we replace $U$ with $t_{i,r} \cdot U $ using Lemma~\ref{lem:permute-row-col},
    so we may assume that $i=r$. 
    Let $w':=\del_{r,*}(w)$. Then, by the row version of Lemma~\ref{lem:containmentcol},
    under the identification of $\iota_{r,*}$, $\mathcal{R}_{w'}$
    is contained in $\mathcal{R}_w$. 
    Therefore $\iota_{r,*}^{-1}(U \setminus \{(r,j)\})$ is $\mathcal{R}_{w'}$-avoiding.
    By the induction hypothesis it is $X_{w'}$-completable, which means
    that there exists some $S'$ such that $(S')^c\supset \iota_{r,*}^{-1}(U \setminus \{(r,j)\})$, 
    and $\llbracket z^{S'} \rrbracket \mathfrak{F}_{w'}(z)>0$.
    Let $S:=\iota_{r,*}(S')\cup ((\Phi(w))_{r,*} \setminus \{(r,j)\})$.
    By Proposition~\ref{prop:deletearbitrary} or
    Proposition~\ref{prop:deleteoneoff} (depending on whether $j=w(1)$), $\llbracket z^S \rrbracket \mathfrak{F}_{w}(z)>0$. Since 
    $S^c$ contains $U$, $U$ is $X_w$-completable.

    By similar reasoning, if there is $j\in [w(1)+2,c]$ such that
    $|U_{*,j}\cap\Phi(w)|=1$, we induct using $w':=\del_{*,j}(w)$. 
    
    We may now assume that for all $i\le r$, $|U_{i,*}\cap \Phi(w)|\ge 2$,
    and for all $j\in [w(1)+2,c]$, $|U_{*,j} \cap \Phi(w)|\ge 2$. 
    \begin{claim}
    \label{claim:tree}
        Under these assumptions, 
        if $U_{[r],[w(1),c]}$ is identified
        with a subgraph of the complete bipartite graph with vertex set $[r] \sqcup [w(1),c]$, $U_{[r], [w(1), c]}$ is a path graph, 
        and 
        $|U_{[r],w(1)}|=|U_{[r],w(1)+1}|=1$.
    \end{claim}
    \begin{proof}
    Let $m:=|U_{[r],[w(1),c]}|$ and $k=c-w(1)+1$.
        Using the graph-theoretic interpretation, since
        $U$ is $\mathcal{R}_w$-avoiding, $U_{[r],[w(1),c]}$
        must be identified with a forest.
        Therefore $m\le r+k-1$.
        By assumption, $m\ge 2r$ and $m \ge 2(k-2)$,
        so $m\ge r+k-2$. Therefore $m=r+k-1$ or 
        $m=r+k-2$. If $m=r+k-1$, the only possibility
        is that $U_{[r],[w(1),c]}$ is a path graph,
        and $|U_{[r],w(1)}|=|U_{[r],w(1)+1}|=1$.
        Otherwise, $U_{[r],[w(1),c]}$ is a forest, so it
        either has an isolated vertex (which must be
        one of the columns $w(1)$ and $w(1)+1$), or at least
        four leaves. In either case, at least one row in $[r]$ or
        one column in $[w(1)+2,c]$ must be a leaf,
        contradicting the assumption.
    \end{proof}

    Since we may permute rows in $[r]$ using Lemma~\ref{lem:permute-row-col}, we may assume
    that $(1,w(1))\in U$ and that
    $(i,w(1)+1)\in U$ for some $i\le r$.

    Let $(r_2,c_2)\in \mathrm{ess}(w)\setminus \dom(w)$
    such that $c_2<c$ is maximal and $r_2$ is as small as possible among such boxes. Namely, 
    $(r_2,c_2)$ is the second essential box
    in $\Phi(w)\setminus\dom(w)$, scanning northeast to southwest. (If $(r_2,c_2)$ does not exist, we are done by the reasoning above.)
    We have $r_2>r$. If $c_2<w(1)$,
    then $X_w$ is isomorphic to the product of $X_w\cap \mathbb{A}^{[n]\times [w(1)-1]}$ and
    $X_w\cap \mathbb{A}^{[n]\times [w(1),n]}$, which
    are both smaller matrix Schubert varieties,
    so the result follows by induction. 
    Thus by Lemma~\ref{lem:rank-1-diagram} we assume $c_2=w(1)+1$.
    \medskip
    
    \underline{\textbf{Case II.1}}. There exists some
    $c'<c_2$ such that $(r_2,c')\in D(w)$
    and $(r_2,c')$ is not rank 0. 
    Then we have
    $D(w)_{[r+1,r_2-1],[w(1),w(1)+1]}=\emptyset$. See Figure~\ref{fig:caseII1}.
    \begin{figure}[h]
        \centering
        \includegraphics[width=0.3\linewidth]{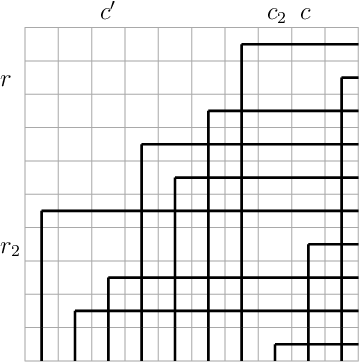}
        \caption{Case II.1}
        \label{fig:caseII1}
    \end{figure}

    \begin{claim} We may assume $U$ contains
    $[r+1,r_2]\times \{w(1)+1\}$.
    \end{claim}
    \begin{proof}
        If $U\supset [r+1, r_2]\times \{w(1)\}$, we may
replace $U$ with $U\cdot s_{w(1)}$ by Lemma~\ref{lem:permute-row-col}. If $U$ contains neither $[r+1, r_2]\times \{w(1)\}$
nor $[r+1, r_2]\times \{w(1)+1\}$, by Lemma~\ref{lem:cotrans-column-ok}, we may replace $U$ with $U\cup ([r+1,r_2]\times \{w(1)+1\})$, since if $f\in\mathcal{R}_w$ has
$\operatorname{spr}(f)_{[r+1,r_2], \{w(1)+1\}}\neq\emptyset$,
there is $h\in \operatorname{cdg}(f)$ such that $h$ is obtained from a Fulton generator of size $r_2-r+1$ whose spread lies in  $[r_2]\times [c_2]$ and must contain $[r+1,r_2]\times \{w(1)\}$.
    \end{proof}

    We can cotransition at $1$, and $w' := s_{w(1)}w$ is in $L_1(w)$, i.e., it appears
    on the right-hand side of the cotransition
    formula. We show that 
    $U\setminus \{(1,w(1))\}$ is $\mathcal{R}_{w'}$-avoiding. 

    Suppose to the contrary that there is $f\in\mathcal{R}_{w'}$ which has
    $\operatorname{spr}(f)\subset U\setminus\{(1,w(1))\}$.
    We will construct a relation in $\mathcal{R}_w$ whose
    spread is contained in $U$.
    If $f$ satisfies
    $\operatorname{spr}(f)\cap [r+1,r_2]\times \{w(1)\}=\emptyset$, then $f\in \mathcal{R}_w$. Since $(r_2,w(1))\in \operatorname{ess}(w')$
    has  rank $r_2-r-1$,
    $\operatorname{spr}(f)\supset [r+1,r_2]\times \{w(1)\}$. Either $f\in\mathrm{CDG}(w')$, in which case we set
    $h:=f$,
    or we have $f=\mathrm{merge}_{(\alpha,\beta)}(f_1,h)$,
    where $f_1\in \mathcal{R}_{w'}$, $h\in \mathrm{CDG}(w')$ and 
    $\operatorname{spr}(h)\supset [r+1,r_2]\times \{w(1)\}$. Note that
    since $\operatorname{spr}(f_1)\cap [r+1,r_2]\times \{w(1)\}=\emptyset$, we have
    $f_1\in\mathcal{R}_w$.
    Let $\widetilde{h}\in\mathrm{CDG}(w)$ 
    be the (unique) generator such that
    $\operatorname{spr}(\widetilde{h})\supset \operatorname{spr}(h) \cup ([r+1,r_2]\times \{w(1)+1\}) \cup \{(1,w(1))\}$. 
    If $f=h$, let $\widetilde{f}:=\widetilde{h}$; otherwise
    let $\widetilde{f}:=\mathrm{merge}_{(\alpha,\beta)}(f_1,\widetilde{h})$.
    Then if $(1,w(1)+1)\in U$, we have $\operatorname{spr}(\widetilde{f})\subset U$, a contradiction. Otherwise, $\operatorname{spr}(\widetilde{f})\setminus \{(1,w(1)+1)\}\subset U$.
    By Claim~\ref{claim:tree}, there exists $f_2\in \mathcal{R}_w$ with $\operatorname{spr}(f_2)\subset U_{[r],[w(1)+1,c]} \cup \{(1,w(1)+1)\}$. Then $\mathrm{merge}_{(1,w(1)+1)}(\widetilde{f}, f_2)$ lies in $\mathcal{R}_w$ and its spread
    is contained in $U$.

    By induction, using Proposition~\ref{prop:cotransition},
    we conclude that $U$ is $X_w$-completable.
\medskip

    \underline{\textbf{Case II.2}}. 
    For all $c'<c_2$, if $(r_2,c')\in D(w)$,
    then $(r_2,c')\in\dom(w)$.
    Let $r'>r$ denote the smallest row
    index such that $(r',c_2)\in D(w)$. See
    Figure~\ref{fig:caseII2}.
\begin{figure}[h]
    \centering
    \includegraphics[width=0.3\linewidth]{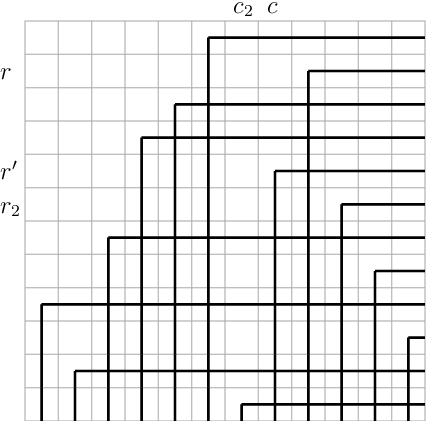}
    \caption{Case II.2}
    \label{fig:caseII2}
\end{figure}
    
    Let $m:=r'-r-1$. By Lemma~\ref{lem:high-rank-diagram},
    we have $w([r+1,r'-1])=[w(1)-m, w(1)-1]$.
    Let $V:=([r+1, r']\times [w(1)-m, w(1)+1])\setminus\dom(w)$.

    We argue that we cannot have $V \subset U$. 
    Suppose otherwise.
    Let $h\in \mathrm{CDG}(w)$ such that
    $\operatorname{spr}(h)= V\cup \{(1,w(1)),(1,w(1)+1)\}$. Since $U$
    is $\mathcal{R}_w$-avoiding and $(1,w(1))\in U$,
    it must be the case that $(1,w(1)+1)\not\in U$.
    However, by Claim~\ref{claim:tree}, there is $f\in \mathcal{R}_w$ with $\operatorname{spr}(f)\subset U_{[r],[w(1)+1,c]} \cup \{(1,w(1)+1)\}$. Then $\mathrm{merge}_{(1,w(1)+1)}(h, f)$ lies in $\mathcal{R}_w$ because it can be obtained by repeatedly merging CDG generators, and its spread
    is contained in $U$.

    Therefore, there exists $(i,j)\in V$ such that $(i,j)\not\in U$. 
    We may also assume that 
    $|(U\cap\Phi(w))_{i,*}|=|V_{i,*}|-1$,
    since if the spread of some $f\in\mathcal{R}_w$ intersects
    $V_{i,*}$ it must contain $V_{i,*}$.
    Let $w':=\del_{i,*}(w)$. If $i=r'$
    and $j=w(1)+1$, we apply Proposition~\ref{prop:deletearbitrary};
    if $i=r'$ and $j<w(1)+1$, we apply
    Proposition~\ref{prop:deleteoneoff};
    if $i<r'$, we apply Proposition~\ref{prop:deletesolid}.
    The claim then follows by induction.
\end{proof}

\bibliography{rank.bib}
\bibliographystyle{amsalpha}

\end{document}